\documentclass[11pt,reqno]{article}
\usepackage{epsfig,amssymb,latexsym,amsmath,pifont,multicol,mathtools,verbatim,amsthm,float,lipsum}
\usepackage{caption} 
\usepackage[utf8]{inputenc}
\usepackage[T1]{fontenc}

\usepackage[varg]{txfonts}
\let\mathbb=\varmathbb

\usepackage[sans]{dsfont}

\usepackage[left=2cm, right=2cm, top=2cm, bottom=2cm]{geometry}

\usepackage[%
cal=euler,
bb=fourier,
scr=zapfc,
]
{mathalfa}

\newcommand{\R}{\mathbb{R}}

\newcommand{\N}{\mathbb{N}}
\newcommand{\tto}{\rightrightarrows}
\newcommand{\gph}{\operatorname{gph}}
\newcommand{\dom}{\operatorname{dom}}

\renewcommand{\epsilon}{\varepsilon}

\newcommand{\cl}{{\text{cl }}}

\DeclareMathOperator{\proj}{proj}
\DeclareMathOperator{\aff}{aff}
\newcommand{\Aver}[2]{\overline{#1}_{#2}}
\usepackage[font=small,labelfont=bf]{caption}
\usepackage{subfigure}
\usepackage{graphicx}
\graphicspath{{Figures/}} 
\usepackage{acronym}
\usepackage{latexsym}
\usepackage{paralist}
\usepackage{wasysym}
\usepackage{xspace}
\usepackage{framed}
\usepackage{palatino,pxfonts}
\usepackage{authblk}
\usepackage{enumitem}
\usepackage[numbers,sort&compress]{natbib}

\usepackage[dvipsnames,svgnames]{xcolor}
\colorlet{MyBlue}{DodgerBlue!75!Black}
\colorlet{MyGreen}{DarkGreen!95!Black}

\usepackage{hyperref}
\hypersetup{
colorlinks=true,
linktocpage=true,
pdfstartview=FitH,
breaklinks=true,
pdfpagemode=UseNone,
pageanchor=true,
pdfpagemode=UseOutlines,
plainpages=false,
bookmarksnumbered,
bookmarksopen=false,
bookmarksopenlevel=1,
hypertexnames=true,
pdfhighlight=/O,
urlcolor=MyBlue!60!black,linkcolor=MyBlue!70!black,citecolor=DarkGreen!70!black, % <--- for screen
pdftitle={},
pdfauthor={},
pdfsubject={},
pdfkeywords={},
pdfcreator={pdfLaTeX},
pdfproducer={LaTeX with hyperref}
}
\numberwithin{equation}{section}  %numberwithin goes before cleverefs when using hyperref
\usepackage[sort&compress,capitalize,nameinlink]{cleveref}
\crefname{example}{Ex.}{Exs.}

\crefrangeformat{equation}{\upshape(#3#1#4)\textendash(#5#2#6)}

\usepackage{algorithm}
\usepackage{algpseudocode}
\theoremstyle{plain}
\newtheorem{theorem}{Theorem}
\newtheorem{corollary}[theorem]{Corollary}
\newtheorem*{corollary*}{Corollary}
\newtheorem{lemma}[theorem]{Lemma}
\newtheorem{proposition}[theorem]{Proposition}

\theoremstyle{definition}
\newtheorem{definition}[theorem]{Definition}
\newtheorem*{definition*}{Definition}
\newtheorem*{problem*}{Problem}
\newtheorem{assumption}{Assumption}
\theoremstyle{remark}
\newtheorem{remark}{Remark}
\newtheorem*{remark*}{Remark}
\newtheorem*{notation*}{Notational remark}
\newtheorem{example}{Example}

\numberwithin{theorem}{section}
\numberwithin{remark}{section}
\numberwithin{example}{section}

\title{Stochastic Differential Inclusions driven by Maximal Monotone Operators with empty interiors}
\date{\today}

\author[1]{\small Juan Guillermo Garrido}
\author[1]{\small Pedro Pérez-Aros} 
\author[2]{\small Mathias Staudigl}

\affil[1]{\footnotesize Departamento de Ingenier\'ia Matem\'atica,  Universidad de Chile, Santiago, Chile\\
(\href{mailto:jgarrido@dim.uchile.cl}{jgarrido@dim.uchile.cl},\href{mailto:pperez@dim.uchile.cl}{pperez@dim.uchile.cl}) }
\affil[2]{\footnotesize Mannheim University, Department of Mathematics, B6 26, 68159 Mannheim, Germany\\
(\href{mailto:mathias.staudigl@uni-mannheim.de}{m.staudigl@uni-mannheim.de})}

\begin{document}

\maketitle

\begin{abstract}%
  This paper studies the long-time behavior of stochastic differential inclusions driven by maximal monotone operators, motivated by continuous-time models of first-order optimization methods under noisy or approximate operator information. We first address well-posedness and show that existence and uniqueness can be established without the customary requirement that the operator’s domain has nonempty interior, by adopting an appropriate notion of solution. We then analyze asymptotic properties of the resulting stochastic dynamics, extending convergence guarantees beyond previously studied settings that rely on smooth potentials, full-domain subdifferentials, or Lipschitz monotone operators. In addition, we consider a  Tikhonov-type regularization of the stochastic inclusion and prove corresponding well-posedness and long-time convergence results.
\end{abstract}

% REQUIRED


\section{Introduction}
Let $f\colon\R^{d}\to\R\cup\{\infty\}$ be a proper, convex, and lower semicontinuous function.
A central continuous-time model in convex optimization is the \emph{subgradient flow}
\begin{equation}\label{subgradient-system1}
    -x'(t)\in \partial f(x(t)),
    \qquad x(0)=x_0\in\dom f,
\end{equation}
where $\partial f$ denotes the convex subdifferential. In finite dimension, solutions of \eqref{subgradient-system1}
are globally defined and, whenever the function attains its minimum, converge as $t\to\infty$ to minimizers of $f$.
More broadly, \eqref{subgradient-system1} is a particular instance of evolution inclusions governed by maximal monotone
operators, whose well-posedness and long-time behavior have been extensively studied
(see, e.g., \cite{MR755330,MR348562,MR2731260}).

This dynamical viewpoint is also tightly connected to numerical algorithms from mathematical optimization. Indeed, classical time discretizations of
maximal-monotone evolutions yield fundamental first-order methods, such as implicit schemes and splitting procedures
based on resolvents and proximal mappings. In turn, qualitative properties of the continuous-time trajectories
(Lyapunov decrease, asymptotic convergence, stability) provide a useful blueprint for interpreting convergence
mechanisms and, in favorable regimes, rates for broad classes of descent algorithms (see, e.g.,
\cite{MR1398330,MR2593044}, and the references therein).

Modern large-scale applications often involve uncertainty, where the driving first-order information is noisy,
incomplete, or obtained through sampling. In such settings, exact evaluations of $\partial f$ (or, more generally, of a
monotone operator) may be prohibitively expensive, while stochastic approximations are computationally viable.
This motivates stochastic counterparts of maximal-monotone dynamics, which can be viewed as continuous-time version
of stochastic first-order schemes (see, e.g., \cite{MR3361444}).

In this work we study the long-time behavior of solutions to the stochastic differential inclusion (SDI)
\begin{equation}\label{MainSDI}
    dX_t \in -A(X_t)\,dt + \sigma(t,X_t)\,dB_t, 
    \qquad X_0 \in \cl \dom A,
\end{equation}
where $A\colon\R^d\rightrightarrows\R^d$ is a maximal monotone operator, $\sigma\colon [0,T]\times \R^d \to \R^{d\times\ell}$
is a given diffusion coefficient, and $B$ denotes an $\ell$-dimensional Brownian motion.
A first fundamental issue in this multivalued setting is well-posedness. Existing approaches typically impose structural
conditions that allow one to adapt classical SDI techniques; in particular, it is customary to assume that $\dom A$ has
nonempty interior (see, e.g., \cite{MR1442017,MR1626174,MR2380366}).
Our first contribution is to show that this interiority assumption can be removed by adopting a modified notion of
solution. This yields existence and uniqueness results for general SDI without requiring any additional assumption of the maximal monotone operator $A$.

A second, and central, question concerns the asymptotic behaviour of \eqref{MainSDI}.
Recent works have established convergence results primarily in settings where the operator has additional structure.
Most contributions address the smooth case $A=\nabla f$, with $f$ convex and differentiable, and show that several
convergence properties of the deterministic dynamics extend to the stochastic setting (see, e.g.,
\cite{MR3747470,MR4989859}).
In the nonsmooth regime, \cite{MR4991707} treats the case $A=\partial g$ for convex, continuous
functions $g$ with full domain, while related results for Lipschitz continuous monotone operators are obtained in
\cite{MR4845874}. The aim of the present work is to go beyond these frameworks and establish convergence results
for \eqref{MainSDI} driven by a general maximal monotone operator.

Finally, we also analyze a Tikhonov-type regularization of \eqref{MainSDI}, namely
\begin{equation}\label{tikhonov-system}
    dX_t \in -A(X_t)\,dt - \epsilon(t) X_t\,dt + \sigma(t,X_t)\,dB_t,
    \qquad X_0 \in \cl \dom A,
\end{equation}
where $\epsilon(\cdot)$ is a nonnegative parameter function. This additional term plays the role of a vanishing
regularization that is relevant both for asymptotic selection mechanisms and for algorithmic stabilization.
Building on our well-posedness theory for general SDIs, we establish existence and uniqueness for
\eqref{tikhonov-system} and derive its long-time convergence properties alongside those of
\eqref{MainSDI}.

The recent paper \cite{luke2026asymptotic} is very much related to this work. They study a penalty-regulated stochastic differential inclusion and prove that the trajectories of the constructed dynamical system exhibit asymptotic convergence towards a solution of constrained variational inequality. This result is obtained under the standard setting, assuming that the driving maximally monotone operator in the drift has non-empty interior. The present paper complements thus the findings of \cite{luke2026asymptotic} by proving existence and uniqueness of solution for a much more general class of operators. Moreover, we perform in this paper the analysis of the Tikhonov regularized dynamical system, which behaves dynamically very differently to penalty-governed dynamics.  

 The rest of the work is organized as follows: In Section \ref{Section2}, we introduce the notation, definitions, and auxiliary lemmas used throughout the work. Section \ref{Section3} extends the existence and uniqueness theory for a general SDI, without any assumption on the maximal monotone operator (in fact, we not require the non-emptiness of the interior of the domain). In Section \ref{Section4}, we establish general convergence results for the solutions of \eqref{MainSDI} under mild decaying assumptions on $\sigma$, which extends the convergence results obtained in \cite{MR4989859}. In particular, we prove convergence of the averaged process and show that, whenever the associated deterministic dynamics (i.e., when $\sigma = 0$ in \eqref{MainSDI}) converges to a zero of $A$, the stochastic system \eqref{MainSDI} converges as well. We also provide rates of convergence in some special cases. Section \ref{Section5} focuses on the case $A = \partial \varphi$, with $\varphi$ being a proper, convex, and lower semicontinuous function, and provides convergence rates in general settings, extending the study of \cite{MR4989859, MR3747470} in the smooth case ($\varphi\in \mathcal{C}^{1,+}$ or $\mathcal{C}^2$) to the nonsmooth setting. Finally, in Section \ref{Section6}, we study the convergence of solutions to \eqref{tikhonov-system} without additional assumptions on $A$ and under mild hypotheses on the parameter function $\epsilon$.
\section{Notation and Preliminaries in Convex Analysis}\label{Section2}

In this section we introduce the notation and some preliminary results that will be used throughout the paper. 
The Euclidean space $\mathbb{R}^d$ is equipped with the standard inner product and the associated norm
\[
  \langle x,y\rangle = x^\top y, 
  \qquad \|x\| = \sqrt{\langle x,x\rangle}.
\]
For a matrix $P \in \mathbb{R}^{m \times n}$, its \emph{adjoint} (transpose) is denoted by $P^\ast$. 
We write $\|P\|$ for the Frobenius norm of $P$.

\subsection{Notions in Convex Analysis}

For a set $C \subseteq \mathbb{R}^d$, the \emph{affine hull} of $C$ is defined as
\[
  \operatorname{aff}(C) := \left\{ \sum_{i=1}^m \lambda_i x_i : x_i \in C,\ \lambda \in \mathbb{R}^m,\ \sum_{i=1}^m \lambda_i = 1 \right\},
\]
that is, the smallest affine subspace containing $C$. Equivalently, if $x_0 \in C$, then
\[
  \operatorname{aff}(C) = x_0 + \operatorname{span}(C - x_0),
\]
where $C - x_0 := \{x - x_0 : x \in C\}$ and $\operatorname{span}(C - x_0)$ denotes the linear subspace generated by $C - x_0$. It follows by definition that
\begin{equation}\label{eqaff}
\operatorname{aff}(C) - a = \operatorname{aff}(C-C) = \operatorname{span}(C - C) 
\quad \forall a \in C .
\end{equation}
The \emph{convex hull} of a set $C \subseteq \mathbb{R}^d$ is defined as
\[
    \operatorname{conv}(C) := \left\{ \sum_{i=1}^m \lambda_i x_i : x_i \in C,\ \lambda_i \geq 0,\ \sum_{i=1}^m \lambda_i = 1 \right\},
\]
that is, the smallest convex set containing $C$. For a set $C \subseteq \R^d$ (not necessarily convex), the \emph{relative interior} of $C$ is defined as the interior of $C$ relative to its affine hull, namely,  
\[
    \operatorname{rint}(C) := \{x \in C : \exists \,\varepsilon > 0 \text{ such that } \mathbb{B}_\varepsilon(x) \cap \operatorname{aff}(C) \subseteq C \}.
\]
It is well-known that every convex set in finite dimensions has a nonempty relative interior (see, e.g., \cite[Proposition 2.40]{MR1491362}). More generally, if a set $U$ is nearly convex in the sense that there exists a convex set $C$ with $C \subseteq U \subseteq \cl C$, then $\operatorname{rint}(U) \neq \varnothing$. In particular, for our purposes, the domains of maximal monotone operators are nearly convex sets, and hence their relative interiors are nonempty (see, e.g., \cite[Theorem 12.41]{MR1491362}).

\noindent If $C \subset \mathbb{R}^m$ is a nonempty subset and $v \in \mathbb{R}^m$, we say that $v$ is \emph{orthogonal} to $C$, denoted $v \perp C$, if 
\[
  \langle v, x-y \rangle = 0, \qquad \forall\, x,y \in C,
\]
that is, $v$ is orthogonal to the space generated by $C - C$. If $L\colon \mathbb{R}^d \to \mathbb{R}^m$ is a linear operator, its \emph{range} (or image) is defined as
\[
    \operatorname{rge}(L) := \{ Lx : x \in \mathbb{R}^d \}.
\]
Given a set $C \subset \R^d$, and a point $x\in \R^d$, we denote  $d(x;C)$   the distance from $x$ to $C$. 
\iffalse
A function $f\colon \mathbb{R}^d \to \R\cup\{\infty\}$ is called \emph{convex} if
\[
  f(\theta x + (1-\theta)y) \leq \theta f(x) + (1-\theta) f(y), 
  \qquad \forall\, x,y \in \mathbb{R}^d,\ \theta \in [0,1].
\]
\fi 
For a function $f\colon \mathbb{R}^d \to \R\cup\{\infty\}$ the \emph{effective domain} of $f$ is given by
\[
  \operatorname{dom} f := \{x \in \mathbb{R}^d : f(x) < \infty\}.
\]
The \emph{subdifferential} of a convex function $f$ at $x \in \operatorname{dom} f$ is defined as
\[
  \partial f(x) := \{ v \in \mathbb{R}^d : f(y) \geq f(x) + \langle v, y - x \rangle \ \ \forall y \in \mathbb{R}^d \}.
\]
Any element $v \in \partial f(x)$ is called a \emph{subgradient} of $f$ at $x$.  

\noindent An operator $A\colon \mathbb{R}^d \rightrightarrows \mathbb{R}^d$ is called \emph{monotone} if
\[
  \langle u-v, x-y \rangle \geq 0, \qquad \forall (x,u), (y,v) \in \operatorname{gph} A,
\]
where the \emph{graph} of $A$ is
\[
  \operatorname{gph} A := \{(x,u) \in \mathbb{R}^d \times \mathbb{R}^d : u \in A(x)\}.
\]
The operator $A$ is said to be \emph{maximal monotone} if it is monotone and its graph cannot be properly extended while preserving monotonicity.  The resolvent of $A$ with parameter $\lambda>0$ is defined as $J_{\lambda A}(x) := (\text{Id} + \lambda A)^{-1}$. When $A$ is maximal monotone, we know that $J_{\lambda A}$ is single-valued on $\R^d$ (see, e.g., \cite[Theorem 21.1]{MR3616647}). For $x\in \dom A$, $A(x)$ is convex and closed (see \cite[Proposition 20.36]{MR3616647}) and we denote by $A^0(x)$ to the element of $A(x)$ with minimal norm.

\noindent For the sake of simplicity of the presentation, all the previous definitions and notations are understood to hold in any arbitrary finite-dimensional Hilbert space.

\noindent We recall the following well-known result, which provides a sufficient condition for the preservation of maximal monotonicity under affine transformations (see, e.g., \cite[Theorem~12.43]{MR1491362}, \cite[Theorem 4]{MR2577332}, \cite[Theorem 6]{MR2160665}, or \cite[Corollary 4.4]{MR1774774}).  

\begin{theorem}[Maximal monotonicity under composition]\label{thm:max_mono_comp}
Let $X,Y$ be two finite dimensional Hilbert spaces, and $A \colon  Y \rightrightarrows Y$ be  maximal monotone mapping, $P \colon  X \to Y$ a be a linear mapping. Define $T \colon  X \tto X$ by 
\[
  T(x) := P^{\ast} A(Px).
\]
If $  \operatorname{rge} P \cap \operatorname{rint} \left(  \dom A  \right)  \neq \emptyset$, then $T$ is maximal monotone.
\end{theorem}

\subsection{Measure theory and stochastic processes}
In what follows we work on a filtered probability space 
$(\Omega,\mathcal{F},(\mathcal{F}_t)_{t\ge 0},\mathbb{P})$ 
satisfying the standard conditions: the filtration is right-continuous 
($\mathcal{F}_t=\bigcap_{u>t}\mathcal{F}_u$ for all $t\ge 0$) and complete 
(every $\mathbb{P}$-null set in $\mathcal{F}$ belongs to $\mathcal{F}_0$). Following standard conventions and notation, we suppress the dependence on  
$\omega \in \Omega$ for convenience when dealing with random variables or 
stochastic processes. %This keeps the notation cleaner and avoids unnecessary overload.

A stochastic process with values in $\mathbb{R}^d$ (equipped with its Borel 
$\sigma$-algebra $\mathcal{B}(\mathbb{R}^d)$) is a family $X=(X_t)_{t\ge 0}$ such that 
$X_t\colon \Omega\to \mathbb{R}^d$ is measurable for each $t\ge 0$.

A process $X$ is said to be $(\mathcal{F}_t)_{t\ge 0}$-adapted if $X_t$ is 
$\mathcal{F}_t$-measurable for all $t\ge 0$. 
It is \emph{progressively measurable} (with respect to $(\mathcal{F}_t)_{t\ge 0}$) if, 
for every $t\ge 0$, the map 
\[
   (s,\omega)\mapsto X_s(\omega), \qquad (s,\omega)\in [0,t]\times\Omega,
\] 
is $\mathcal{B}([0,t])\otimes\mathcal{F}_t$-measurable.

\noindent A random time $\tau\colon \Omega\to[0,\infty]$ is a \emph{stopping time} 
(with respect to $(\mathcal{F}_t)_{t\ge 0}$) if $\{\tau\le t\}\in\mathcal{F}_t$ for all 
$t\ge 0$. Given a process $X$,  the \emph{stopped process} is defined by 
\[
   X_t^\tau := X_{t\wedge \tau}, \qquad t\ge 0.
\]
Let $p \in [1, \infty)$. we denote by  $L_{\mathbb{P}}^{p}(\Omega, \mathbb{R}^d)
$   
the set of all (equivalence classes of) $\mathcal{F}$-measurable functions 
$\xi \colon \Omega \to \mathbb{R}^d$ such that
\[
\int_{\Omega} \|\xi(\omega)\|^{p} \, d\mathbb{P}(\omega) < \infty.
\]
Particularly, we denote by $L^p_{\mathrm{loc}}(\mathbb{R}_+;\mathbb{R})$ 
the space of locally $p$-integrable functions, that is,
\[
  L^p_{\mathrm{loc}}(\mathbb{R}_+;\mathbb{R})
  := \left\{ f:\mathbb{R}_+\to\mathbb{R}\ \text{measurable} : 
      \int_0^T |f(t)|^p\,dt < \infty,\ \forall T>0 \right\}.
\]
%For a measurable function $f\colon \Omega\to \R$, we define $$\text{esssup}_\Omega f := \inf\{r\geq 0 : |f(\omega)|\leq r, \mathbb{P}\text{-a.s.}\}.$$
Finally, if $f\colon [a,b]\to \mathbb{R}^d$ is continuous and 
$g\colon [a,b]\to \mathbb{R}^d$ is of bounded variation, we define the 
\emph{Riemann–Stieltjes integral} of $f$ with respect to $g$ by
\[
   \int_a^b \langle f(t) , dg(t)  \rangle 
   := \lim_{\|\pi\|\to 0} \sum_{i=0}^{n-1} \langle f(t_i), g(t_{i+1})-g(t_i)\rangle,
\]
where the limit is taken over partitions 
$\pi=\{a=t_0<t_1<\cdots<t_n=b\}$ of $[a,b]$, provided it exists.
It is well known that this limit exists for every continuous 
$f\colon [a,b]\to \mathbb{R}^d$ and every bounded variation function 
$g\colon [a,b]\to \mathbb{R}$. We refer to \cite{MR957087} for more details.

 \subsection{Integral average and like-Opial’s Lemma}
    For every locally integrable function \\ $x\colon \R_+\to \R^d$, we define its integral average as 
    \begin{equation}\label{Int_Aver}
        \Aver{x}{t} := \frac{1}{t}\int_0^t x(s)ds.
    \end{equation}  
The following lemma is in the spirit of Opial's Lemma for the integral average defined in \eqref{Int_Aver}. See \cite[Lemma 2.3]{MR2593044}. 
    \begin{lemma}\label{opial-prop}
        Consider a continuous trajectory $x\colon [0,\infty[\to \R^d$ and some set $\mathcal{S}\neq\emptyset$ and suppose that for all $x^\star\in \mathcal{S}$, $\displaystyle\lim_{t\to\infty}\|x(t)-x^\star\|$ exists and every cluster point of $\Aver{x}{t}$ belongs to $\mathcal{S}$. Then there is $x^\star\in \mathcal{S}$ such that $\displaystyle\lim_{t\to \infty}\Aver{x}{t} =x^\star$.
    \end{lemma}
 The final three lemmas in this section follow from standard arguments, so we omit their proofs.
 
    \begin{lemma}\label{lem-convergence-zero}
    	Let $T_0\in \R$ and $f\colon [T_0,\infty[\to \R_+$ be a function such that  $\int_{T_0}^{\infty}f(s)ds<\infty$, then $\displaystyle\liminf_{t\to \infty} f(t) = 0$.
    \end{lemma}
%     \begin{proof}
%     By contradiction, suppose that $\displaystyle\liminf_{t \to \infty} f(t) > a > 0$. Then there exists $T > T_0$ such that $f(t) \geq a$ for all $t \geq T$. Consequently, for all $t \geq T$ we have
% \begin{equation*}
%     a(t-T) \leq \int_T^t f(s)\,ds \leq \int_{T_0}^\infty f(s)\,ds < \infty,
% \end{equation*}
% which is a contradiction since the left-hand side diverges as $t \to \infty$. Hence, we conclude that $\displaystyle\liminf_{t \to \infty} f(t) = 0$.
%     \end{proof}
    
    \begin{lemma}\label{jensen-convex}
    	Consider $f\colon \R^d \to \R\cup \{\infty\}$ be a proper, convex, and lower semicontinuous function  and $x\colon \R_+\to \R^d$ be a continuous trajectory, then
    	\begin{equation*}
    		f(\Aver{x}{t} )\leq \frac{1}{t}\int_0^tf(x(s))ds, \text{ for all  } t\in \R_+.
    	\end{equation*}
    \end{lemma}
%     \begin{proof}
%     	Since $f$  is a proper, convex, and lower semicontinuous function, we can write $f = \sup_{i \in I} f_i$, where $\{f_i : i \in I\}$ is a countable family of affine functions. For every $i \in I$ we have
% \begin{equation*}
%     f_i(\Aver{x}{t}) 
%       = \frac{1}{t}\int_0^t f_i(x(s))\,ds
%       \leq \frac{1}{t}\int_0^t f(x(s))ds
% \end{equation*}
% Taking the supremum over $i \in I$ on the left-hand side yields the desired inequality.
% \end{proof}

    \begin{lemma}\label{lemma-comparison}
    	Let $t_0\geq 0$ and $T>0$. Consider an lsc function $\psi\colon [t_0,T]\to \R_+$ satisfying the following inequality
    	\begin{equation*}
    		\forall t>s  : \psi(t)\leq \psi(s) - \int_s^t\ell(\tau)\psi(\tau)d\tau + \int_s^t\beta(\tau)d\tau
    	\end{equation*}
    	where $\ell\colon \R_+\to \R_+$ is continuous and $\beta$ is a locally integrable function. Then 
        \[ \psi(t)\leq \exp\left(-\int_{t_0}^t \ell(s)ds\right)\left[\psi(t_0) + \int_{t_0}^t \exp\left(\int_{t_0}^s\ell(\tau)d\tau\right) \beta(s)\,ds\right].\]
    \end{lemma}
    
%     \begin{proof}
%     	 Define $f(t,x) := \ell(t) x - \beta(t)$. Note that $f(t,\cdot)$ is nondecreasing and continuous on the second variable.  The solution of the differential equation 
% \[
%     y'(t) = -f(t,y(t)), 
%     \qquad y(t_0) = \psi(t_0),
% \]
% is given by
% \begin{equation*}
%     y(t) = \exp\left(-\int_{t_0}^t \ell(s)ds\right)\psi(t_0) + \int_{t_0}^t \exp\left(-\int_s^t\ell(\tau)d\tau\right) \beta(s)\,ds.
% \end{equation*}
% The conclusion then follows from \cite[Proposition~2.3]{MR4150268}.
%     \end{proof}
\section{The SDI Model for Monotone Inclusions and Definition of Solutions}\label{Section3}

In this section, we introduce the stochastic differential inclusion (SDI) model that will be the focus of this work, together with the notion of solution that we shall adopt.  

Let $A \colon \R^d \tto \R^d$ be a maximal monotone operator, and let $(\Omega,\mathcal{F},(\mathcal{F}_t),\mathbb{P})$ be a filtered probability space. We consider the following stochastic differential inclusion:
\begin{equation}\label{MainSDI_General}
    dX_t \in -A(X_t)\,dt+ F(t,X_t)\,dt + \sigma(t,X_t)\,dB_t, 
    \qquad X_0 \in \cl \dom A,
\end{equation}
where $\sigma \colon \Omega\times \R_+ \times \R^d \to \R^{d\times \ell}$, $F\colon \Omega \times \R_+ \times \R^d \to \R^d$ are mappings, and $B$ is an $\ell$-dimensional Brownian motion adapted to the filtration $(\mathcal{F}_t)_{t\geq 0}$.  

We now introduce the concept of solution that will be used in this work.  

 \begin{definition}\label{Def_solutionSDI}
A triplet $(X,Y,M)$ of $\mathbb{R}^d$-valued processes is called a \emph{solution} of \eqref{MainSDI_General} if:
\begin{enumerate}[label=\alph*)]
    \item $X$, $Y$, and $M$ are continuous progressively measurable processes with $Y_0=0$ and $M_0=0$.
    
    \item $X_t \in \overline{\operatorname{dom} A}$ for all $t \geq 0$, $\mathbb{P}$-a.s., the mapping $(\omega,t)\mapsto \sigma(\omega, t,X_t(\omega))$ is progressively measurable, and for all $T>0$
    \begin{align*}
     \int_0^T F(t,X_t) dt < \infty \text{ and }  \int_0^T \|\sigma(t,X_t)\|^2 dt < \infty 
       \quad \mathbb{P}\text{-a.s.}
    \end{align*}

    \item $Y$ has paths of bounded variation on compact intervals. Moreover, for every pair of continuous $(\mathcal{F}_t)$-adapted processes $(\alpha,\beta)$ with $(\alpha_t,\beta_t)\in\operatorname{gph} A$ for all $t \geq 0$, the process
    \begin{equation}\label{mono-Yt}
        \langle X_t - \alpha_t, dY_t - \beta_t dt \rangle
    \end{equation}
    defines a nonnegative measure on $\mathbb{R}_+$. Equivalently, for all $0 \leq a \leq b$,
    \begin{equation}\label{measure-prop-solution}
        \int_a^b \langle X_t - \alpha_t, dY_t \rangle 
        \geq \int_a^b \langle X_t - \alpha_t, \beta_t \rangle dt,
        \qquad \mathbb{P}\text{-a.s.}
    \end{equation}

    \item $M$ is a continuous local martingale such that
    \begin{equation}\label{MPerp}
        M_t \perp \operatorname{dom} A,   \qquad \forall t \geq 0,\ \mathbb{P}\text{-a.s.}
    \end{equation}

    \item The decomposition
    \begin{equation}\label{eqn-XY}
        X_t = X_0-Y_t + M_t + \int_0^t F(s,X_s)ds + \int_0^t \sigma(s,X_s) dB_s, 
        \qquad t \geq 0,
    \end{equation}
    holds $\mathbb{P}$-a.s.
\end{enumerate}
When there is no risk of ambiguity, we will simply say that $X$ is a solution to \eqref{MainSDI_General}.
\end{definition}

 Stochastic differential inclusions has been investigated by several authors (see, e.g., \cite{MR3308895} and the references therein). To the best of our knowledge, however, all existing works rely on the assumption that $\dom A$ has nonempty interior. In that case, the decomposition in \eqref{eqn-XY} involves only the bounded-variation process $Y_t$.  

Our approach relaxes this assumption by allowing the presence of the martingale term $M$ in \eqref{eqn-XY}. From a geometric perspective, $M$ captures random fluctuations in directions orthogonal to the affine hull of $\dom A$, induced by the noise in the system. If we denote by 
\[
    L := \operatorname{span}(\dom A - \dom A), 
\]
 then $M_t \in L^\perp$. Furthermore, since $X_t - X_0 \in L$, it follows from \eqref{eqn-XY} that for every $v \in L^\perp$,
\begin{align}\label{eq001}
   \langle Y_t - \int_0^t F(s,X_s)ds  , v\rangle = \langle M_t , v\rangle + \left\langle \int_0^t \sigma(s,X_s)\,dB_s , v \right\rangle.  
\end{align}
In \eqref{eq001}, the left-hand side is a process of bounded variation, while the right-hand side is a local martingale. By \cite[Theorem~4.8]{MR3497465}, we deduce that $\langle Y_t - \int_0^t F(s,X_s)ds  , v\rangle \equiv 0$. Since $v \in L^\perp$ was arbitrary and $L^\perp$ is separable, we conclude that $$Y_t - \int_0^t F(s,X_s)ds  \in L^{\perp \perp } = L.$$ Particularly, using \eqref{eqaff}, we conclude that 
\begin{align}\label{YbelongAFF}
Y_t - \int_0^t F(s,X_s)ds \in \aff (\dom A) -X_0  \quad  \forall t \geq 0,\; \mathbb{P}\text{-a.s.}
\end{align}
   
Furthermore, multiplying \eqref{eqn-XY} by $P^\ast$, the orthogonal projection onto $L$, we obtain
\begin{equation}\label{eqn-XY_Auxiliar}
X_t
  = X_0 - Y_t
  + \int_0^t  F(s,X_s)\,ds
  + \int_0^t P^\ast \sigma(s,X_s)\,dB_s .
\end{equation}
We will use this auxiliary identity throughout the paper, since it allows us to apply It\^o's formula to the process $X_t$.

It is worth emphasizing that if $\dom A$ has nonempty interior, then condition \eqref{MPerp} directly enforces $M_t \equiv 0$, thereby recovering the classical definition adopted in the literature (see, e.g., \cite{MR3308895}). On the other hand, the following example shows that in the general case the martingale component cannot be omitted.

\begin{example}
Consider the convex function
\[
    f(x) :=
    \begin{cases}
        0, & \text{if } x = 0, \\
        \infty, & \text{if } x \neq 0,
    \end{cases}
\]
and let $A := \partial f$, so that $\dom A=\{0\}$. Take $F\equiv 0$, and  the diffusion coefficient $\sigma(t,x) \equiv 1$.  
In this case, any solution must satisfy $X_t \in \cl{\dom A} = \{0\}$ for all $t \geq 0$. Then \eqref{eqn-XY} reduces to
\[
    0 = -Y_t + M_t + B_t.
\]
By \cite[Theorem~4.8]{MR3497465}, this implies $M_t = -B_t$ and $Y_t \equiv 0$.  
\end{example}

Now, we turn into the existence and uniqueness of   \eqref{MainSDI_General}.

\begin{theorem}\label{Theo:EU}
 Let $A \colon \mathbb{R}^d \rightrightarrows \mathbb{R}^d$ be   a maximal monotone operator. Consider $\sigma \colon \Omega \times  \R_+ \times \mathbb{R}^d \to \mathbb{R}^{d\times \ell}$ and $F\colon  \Omega \times \R_+ \times \R^d \to \R^d$  be   mapping such that $\sigma (\cdot, \cdot,  x)\colon \Omega \times \R_+   \to \mathbb{R}^{d\times \ell}$ and $F(\cdot,\cdot, x) \colon  \Omega \times \R_+ \to \R^d $ are   progressively measurable stochastic process for every $x\in \mathbb{R}^d$. Furthermore, let us suppose that there exists $\kappa \in L^2_{\mathrm{loc}}(\R_+; \R)$  and $\zeta \in L^1_{\mathrm{loc}} (\R_+; \mathbb{R})$ such that  $d\mathbb{P}\otimes dt$-almost every $(\omega,t)\in \Omega \times \mathbb{R}_+$
\begin{enumerate}[label=\alph*)]  
   \item $x \mapsto F(t,x) $ is continuous. 
   \item  $  
        \langle F(t,x) - F(t,y) , x - y \rangle \leq \zeta(t) \| x - y\|^2, \text{ for all } x,y\in \R^d.
        $ 
        \item   For all $r  >0$ we have that $\int_0^t \sup\{ \| F(s,u)\| :\|u\| \leq r\} ds < \infty$.  
        \item  $  
           \|\sigma(  t,x)-\sigma( t,y)\| \leq \kappa(t)\|x-y\|, \quad \forall x,y\in\mathbb{R}^d.
        $  
         \item $\int_0^t \|\sigma (s,0)\|^2 ds<\infty$
       \item For all $r>0$ 
        \begin{align*}
        \lim_{ \delta \to 0^+}\mathfrak{s}(\delta, r,t) =0,  
        \end{align*}
        where $$\mathfrak{s}(\delta,r,t):= \sup\limits_{y \in C([0,t];\R^d), \| y\|_\infty \leq r} \left\{  \int_0^t \| F(s + \delta, y(s) ) - F(s,y(s)) \| ds \right\}$$

\end{enumerate}

If $X_0 $ is an $\mathcal{F}_0$-measurable mapping with $X_0 \in  \cl\operatorname{dom}(A)$, then the SDI \eqref{MainSDI_General} has a unique solution $(X,Y,M)$ in the sense of Definition~\ref{Def_solutionSDI}. 
\end{theorem}

\begin{proof}
Without loss of generality, we may assume that $0 \in \dom A$. Otherwise, we consider the SDI associated with the maximal monotone operator
$\tilde{A}(x) := A(x + u_0)$ for some (deterministic) $u_0 \in \dom A$, together
with the corresponding translated mappings
$\tilde{F}(t,u) := F(t,u + u_0)$ and
$\tilde{\sigma}(t,u) := \sigma(t, u + u_0)$ and the initial condition $\tilde{X}_0 = X_0 - u_0$.
In this case, 
\[
    \operatorname{span}(\dom A) = \aff(\dom A).
\] 
Let $L := \operatorname{span}(\dom A)$, which is a Hilbert subspace of $\R^d$,  and let $P\colon  L \to \R^d$  the linear injection  from $L$ onto $\R^d$, which can be identified with a matrix since $L$ is a linear subspace.   It follows that $P^\ast \colon \R^d \to L$. Then, as a consequence of Theorem~\ref{thm:max_mono_comp}, the operator $T\colon  L \tto L$ given by
\[
    T(x) := P^\ast A(Px)
\]
is maximal monotone. Here, $P^\ast \colon \mathbb{R}^d \to L$ is nothing more than the projection onto $L$. Moreover, it is known that $\operatorname{rint}(\dom A) \neq \emptyset$ (see, e.g., \cite[Theorem~12.41]{MR1491362}).   By construction of $T$, we have
\[
    \operatorname{rint}(\dom A) = \operatorname{int}(\dom T),
\]
which implies that $\operatorname{int}(\dom T) \neq \emptyset$. Furthermore, it follows that $X_0 \in \cl(\dom T)$. 

Next, consider the auxiliary stochastic differential inclusion
\begin{equation}\label{MainSDI_aux}
    d\hat{X}_t \in -T(\hat{X}_t)\,dt + \hat{F}(t,X_t)dt +  \hat{\sigma}(t,\hat{X}_t)\,dB_t, 
    \qquad \hat{X}_0 \in \cl(\dom T),
\end{equation}
where $\hat{\sigma} (s,x) $ is the linear function defined by $\tilde{\sigma} \colon \Omega\times \R_+ \times L \to \mathcal{L} (\R^\ell, L)$ given by $\hat{\sigma} (s, x):= P^\ast\sigma(s,x) = P^\ast\sigma(s,Px) $ and $\hat{F} \colon \Omega\times \R_+ \times L \to L $ given by   $\hat{F}(t,x) =P^\ast F(t,x)$.

By \cite[Theorem~4.19]{MR3308895}, the SDI \eqref{MainSDI_aux} admits a unique solution $(\hat{X}, \hat{Y}, \hat{M})$ in the sense of \cite[Definition 4.2]{MR3308895}, which coincides with   Definition~\ref{Def_solutionSDI} for this setting, since   $\operatorname{int}(\dom T) \neq \emptyset$, which implies  that $\hat{M} \equiv 0$.  

Now, define the progressively measurable processes
\begin{align*}
   X_t   := \hat{X}_t,\; 
     Y_t := \hat{Y}_t+ \int_0^t f(s) ds, \text{ and }
    \; M_t  := \int_0^t \left( \hat{\sigma}(s,\hat{X}_s) -   \sigma(s,\hat{X}_s)\right) dB_s,
\end{align*}
 where $f(t):=    {F}(t,\hat{X}_t) -  \hat{F}(t,\hat{X}_t)$.
It follows that $(X,Y,M)$ is a solution of \eqref{MainSDI_General} in the sense of Definition~\ref{Def_solutionSDI}. Most of the requirements in Definition~\ref{Def_solutionSDI} are straightforward to verify; the only nontrivial point is condition~\eqref{mono-Yt}. So take $(\alpha_t,\beta_t)\in \gph A$. Since $(\alpha_t, P^\ast \beta_t)\in \gph T$, we obtain that

\begin{align*}
\int_{a}^b \langle X_t - \alpha_t , dY_t - P^\ast \beta_t dt\rangle   = & \int_{a}^b \langle X_t - \alpha_t , dY_t dt\rangle  - \int_{a}^b \langle X_t - \alpha_t , P^\ast \beta_t dt\rangle  \\=  &\int_{a}^b \langle X_t - \alpha_t , dY_t dt\rangle  - \int_{a}^b \langle X_t - \alpha_t , \beta_t dt\rangle   \\ =& \int_{a}^b \langle X_t - \alpha_t , d\hat{Y}_t dt\rangle    + \int_{a}^b \langle X_t - \alpha_t , f(t)\rangle dt \\ &  - \int_{a}^b \langle X_t - \alpha_t , \beta_t dt\rangle \\
 =&\int_{a}^b \langle X_t - \alpha_t , dY_t - \beta_t dt\rangle \geq 0,
\end{align*}  
and that shows the existence of the solution.

We now prove uniqueness.  
Let $(X^i,Y^i,M^i)$, $i=1,2$, be two solutions of the SDI \eqref{MainSDI_General}.  
Using \eqref{eqn-XY_Auxiliar}  we get
\begin{equation}\label{eqwithProj}
    X^i_t = X_0 - Y^i_t + \int_0^t F(s,X_s^i) ds +  \int_0^t P^\ast \sigma(s,X^i_s)\,dB_s, 
    \qquad i=1,2. 
\end{equation} 
Furthermore, let us notice that by \cite[Proposition 6.17]{MR3308895}  we have that for all $0\leq a \leq b$
\begin{align*}
    \int_a^b \langle X_t^1 - X_t^2 , dY_t^1 - dY_t^2 \rangle \geq 0, \mathbb{P}\text{-a.s. } 
\end{align*}
Hence, using Itô’s formula, and taking expectations, we deduce that there exists a constant $C>0$ such that 
\[
      \mathbb{E}\|X^1_t - X^2_t\|^2 
      \leq C \int_0^t \left( \zeta(s) +  \kappa(s)^2 \right) \mathbb{E}\|X^1_s - X^2_s\|^2 \, ds.
\]
By Grönwall's inequality, this yields $X^1_t = X^2_t$.  
Substituting into \eqref{eqwithProj}, we further obtain $Y^1_t = Y^2_t$.  
Finally, recalling \eqref{eqn-XY},  it follows that $M^1_t = M^2_t$.
This proves the uniqueness of the solution.
\end{proof}

\begin{example}
    Take the maximal monotone operator $A\colon \R^2\tto \R^2$ given by $$A(x,y) =\begin{cases}
    \{x\}\times \R : y=0,\\
    \emptyset : y\neq 0.
    \end{cases}$$ This operator is the subdifferential of the convex function $$f(x,y) = \begin{cases}\frac{1}{2}x^2 : y=0,\\
    \infty : y\neq 0.
    \end{cases}$$ which is proper, convex, and lower semicontinuous function. Note that $\dom A = \R\times \{0\}$, which is closed and has empty interior. Let $\sigma(t,x) = \frac{1}{1+t}$, $F = 0$ and choose an initial point $(x_0,0)\in \dom A$. The theorem above then yields processes $(X,Y)$ satisfying \begin{equation*}
        (X_t^1,X_t^2) = (x_0,0) - (Y_t^1,Y_t^2) + (M_t^1,M_t^2) + \left(\int_0^t\sigma(s)dB_s^1,\int_{0}^t \sigma(s) dB_s^2\right),
    \end{equation*}
    where $M := (M^1,M^2)$ is a martingale orthogonal to $\dom A$, hence $M_t^1\equiv 0$. Because $X_t\in \dom A$ for all $t$, we must have $X_t^2\equiv 0$. Consequently, $Y_t^2 = M_t^2 + \int_{0}^t \sigma(t) dB_t^2$. Since $Y^2$ has bounded variation while the right hand side is a martingale, we conclude that $Y_t^2 = 0$ and $M_t^2 = -\int_{0}^t \sigma(s) dB_s^2$. Next, observe that $X_t^1 = x_0 - Y_t^1 + \int_0^t \sigma(s)dB_s^1$ and for all continuous processes $(\alpha_t,\beta_t)\in \gph A$ we have 
    \begin{equation*}
        \forall 0\leq a\leq b: \int_a^b \langle X_t-\alpha_t,dY_t\rangle\geq \int_a^b\langle X_t-\alpha_t,\beta_t\rangle dt
    \end{equation*}
    and it can be reduced to
    \begin{equation*}
        \forall 0\leq a\leq b: \int_a^b (X_t^1-\alpha_t^1)dY_t^1\geq \int_a^b (X_t^1-\alpha_t^1)\alpha_t^1dt
    \end{equation*}
    where $\alpha_t^1\colon \R\to \R$ is any continuous function. Choosing $\alpha_t^1 := X_t^1 - c$ with $c\in \R$, we obtain 
    \begin{equation*}
        \forall t\geq 0, \forall c\in \R : cY_t^1\geq c\int_0^t (X_s^1-c)ds
    \end{equation*}
    and inspecting the cases $c>0$ and $c<0$ shows that the only possibility is $Y_t^1 = \int_0^t X_t^1 dt$ and then $X_t^1 = x_0-\int_0^t X_t^1dt +\int_0^t \sigma(t)dB_t^1$. Solving this linear stochastic equation gives
    \begin{equation*}
        X_t^1 = e^{-t} x_0 + e^{-t}\int_0^t e^{s} \sigma(s)dB_s^1.
    \end{equation*}
\end{example}
\section{SDI related to monotone differential inclusions}\label{Section4}
In this section, we consider the SDI \eqref{MainSDI}, for a given maximal monotone operator $A\colon \mathbb{R}^d \to \mathbb{R}^d$. 
The existence and uniqueness of solutions to the SDI \eqref{MainSDI} can be guaranteed by Theorem~\ref{Theo:EU}. 
Nevertheless, our main goal in this section is to understand the convergence of solutions of such a system toward solutions of the inclusion $0 \in A(x)$. 
To this end, we impose some conditions on the volatility coefficient $\sigma$ in the SDI \eqref{MainSDI}. 
Formally, we assume that

 \begin{assumption}\label{assumption01}
    Consider $\sigma \colon \Omega \times \R_+ \times \mathbb{R}^d \to \mathbb{R}^{d\times \ell}$. We may assume that 
\begin{enumerate}[label=\alph*)]
    \item The mapping $\sigma (\cdot,\cdot,  x): \Omega\times \R_+   \to \mathbb{R}^{d\times \ell}$ is a progresively measurable stochastic process for every $x\in \mathbb{R}^d$.
    \item  There is a function $\kappa\in L^2_{loc}(\R_+;\R)$ such that for all $x,y\in \R^d$, $t\in \R_+$
    \begin{equation*} 
        \|\sigma(t,x)-\sigma(t,y)\|\leq \kappa(t)\|x-y\|.
    \end{equation*}
    \item $\mathbb{E} \left( \int_0^\infty  \sigma^2_\infty (t)  dt\right)  < \infty$, where $$ \sigma_\infty(t) := \sup_{x\in \R^d}\|\sigma(t,x)\| $$
    \end{enumerate}
 \end{assumption}
    It is straightforward to verify that under assumption \eqref{assumption01} the existence and uniqueness of the SDI \eqref{MainSDI} is given by  Theorem~\ref{Theo:EU}.
    
Now, let us turn into the investigation of  the asymptotic behaviour of the unique solution 
of the SDI~\eqref{MainSDI} under Assumption~\eqref{assumption01}. Let us denote
\[
    \mathcal{S} := \{x \in \mathbb{R}^d : 0 \in A(x)\}.
\]
Since $A$ is maximal monotone, the set $\mathcal{S}$ is convex and closed.

    \begin{theorem}\label{main-thm-max-mon}
   Let $A \colon \mathbb{R}^d \rightrightarrows \mathbb{R}^d$  be  a maximal monotone operator with $\mathcal{S}\neq \emptyset$. Consider $\sigma \colon \Omega \times  \R_+ \times \mathbb{R}^d \to \mathbb{R}^{d\times \ell}$  satisfying  Assumption~\ref{assumption01}, let $X_0$ be an $\mathcal{F}_0$-measurable, square-integrable random variable, and let $X$ denote the unique solution of \eqref{MainSDI}. Then,

    	\begin{enumerate}[label = (\alph*)]
            \item For all $x^\star\in \mathcal{S}$ and $t,s\geq 0$ with $s<t$,
    		\begin{equation*}
                \begin{aligned}
                    &\mathbb{E}(\|X_t-x^\star\|^2)\\
                    &\leq \mathbb{E}(\|X_s-x^\star\|^2) -2\mathbb{E}\left(\int_s^t\langle X_\tau-x^\star,dY_\tau\rangle\right)+\mathbb{E} \left(\int_s^t\sigma_\infty(\tau)^2d\tau \right).
                \end{aligned}
    		\end{equation*}
    		\item For all $x^\star\in \mathcal{S}$, $\displaystyle\mathbb{E}\left(\sup_{t\geq 0}\|X_t-x^\star\|^2\right)<\infty$.
    		\item $\mathbb{P}$-a.s., for all $x^\star\in \mathcal{S}$, $\displaystyle\lim_{t\to \infty}\|X_t-x^\star\|$ exists. Moreover, for all  $x^\star\in \mathcal{S}$, 
            \begin{equation}\label{conv-L2-zeroe}
                \displaystyle\lim_{t\to\infty}\mathbb{E}(\|X_t-x^\star\|^2) = \mathbb{E}(\lim_{t\to \infty}\|X_t-x^\star\|^2).
            \end{equation}
    		\item There is a random variable $X^\star\in L^2_{\mathbb{P}}(\Omega;\R^d)$ such that $\mathbb{P}$-a.s. $\displaystyle\lim_{t\to \infty}\Aver{X}{t}  = X^\star$ with $\mathbb{P}(X^\star\in \mathcal{S})=1$, and $\displaystyle\lim_{t\to \infty}\mathbb{E}(\|\Aver{X}{t}-X^\star\|^2) = 0$.
    		% \item If $A$ is $\rho$-strongly maximal monotone with $\rho>0$, then $\mathcal{S} = \{x^\star\}$ and $\mathbb{P}$-a.s. $\displaystyle\lim_{t\to \infty}X_t = x^\star$ and $\displaystyle\lim_{t\to \infty}\mathbb{E}(\|X_t-x^\star\|^2)=0$.
    	\end{enumerate}
    \end{theorem}
    \begin{proof}
Let us denote by $L: =\operatorname{span}(\dom A  - \dom A)$, and denote by $P^\ast \colon\mathbb{R}^d \to L$ the  projection onto $L$. Let  $(X,Y,M)$ be the triplet  of $\mathbb{R}^d$-valued processes, which are the solution of  \eqref{MainSDI}.   Consider $x^\star \in \mathcal{S}$. It follows by observation \eqref{YbelongAFF} that
    \begin{equation*}
    X_t = X_0-Y_t   + \int_0^t P^\ast\sigma(s,X_s) dB_s,  
    \end{equation*}
    Hence, by Itô's formula, we have for $s\in \R_+$ and $t\geq s$
    \begin{equation}\label{eqn-ito11}
        \begin{aligned}
            \frac{1}{2}\|X_t-x^\star\|^2 = & \ \frac{1}{2}\|X_s-x^\star\|^2 + \int_s^t \langle X_\tau-x^\star,dX_\tau\rangle + \frac{1}{2}\int_s^t\|P^\ast\sigma(\tau,X_\tau)\|^2 d\tau\\
             = & \ \frac{1}{2}\|X_s-x^\star\|^2 + \int_s^t \langle X_\tau-x^\star,-dY_\tau\rangle + \frac{1}{2}\int_s^t\|P^\ast\sigma(\tau,X_\tau)\|^2 d\tau \\
             & +\int_s^t \langle X_\tau-x^\star,P^\ast\sigma(\tau,X_\tau)dB_\tau\rangle
        \end{aligned}
    \end{equation}
    Since $0\in A(x^\star)$, we have \eqref{mono-Yt} yields
    \begin{equation*}
    	\mathcal{Y}_t:= \int_0^t\langle X_s-x^\star,dY_s \rangle
    \end{equation*}
    is a nondecreasing process and $\mathcal{Y}_t\geq 0, $  for all $ t\geq 0$. Consider   $\mathcal{X}_t := \frac{1}{2}\|X_t-x^\star\|^2$, then $\mathcal{X}_t\geq 0$, for  $t\geq 0$, and define 
    \begin{align*}
        \xi:=\frac{1}{2}\|X_0 - x^\star\|^2, \quad A_t :=  \frac{1}{2}\int_0^t\|P^\ast\sigma(s,X_s)\|^2ds, 
    \end{align*}
    \begin{align*}
        \mathscr{M}_t := \int_0^t \langle X_s-x^\star,P^\ast\sigma(s,X_s)dB_s\rangle. 
    \end{align*}
    It is clear that $(\mathscr{M}_t)_{t\geq 0}$ is a local martingale, and we have
    \begin{equation*}
       \mathcal{X}_t = \xi + A_t -\mathcal{Y}_t + \mathscr{M}_t, \text{ for all } t\geq 0.
    \end{equation*}
    Since $\sigma_\infty$ is square integrable, we have $\displaystyle\lim_{t\to \infty} A_t<\infty$ $\mathbb{P}$-a.s.,  by using \cite[Theorem 1.3.9]{MR2380366} it holds $\mathbb{P}$-a.s. $\displaystyle\lim_{t\to \infty} \mathcal{X}_t$ and $\displaystyle\lim_{t\to \infty} \mathcal{Y}_t$ exist; Specifically:
    \begin{equation*}
		\mathbb{P}\text{-a.s. }\lim_{t\to\infty} \|X_t-x^\star\|^2 \text{ exists and it is finite, and }\int_0^\infty \langle X_s-x^\star, dY_s\rangle<\infty.
    \end{equation*}
    On the other hand, we have for all $t\geq 0$ and $\mathbb{P}$-a.s.
    \begin{equation*}
        \frac{1}{2}\|X_t-x^\star\|^2 \leq \frac{1}{2}\|X_0-x^\star\|^2 + A_t + \mathscr{M}_t.
    \end{equation*}
    Take the stopping time $\tau_n := \inf\{ t\geq 0 : \|X_t\|\geq n \}$, then we have 
    \begin{equation*}
        \begin{aligned}
            \sup_{s\leq t}\|X_{s\wedge \tau_n}-x^\star\|^2 &\leq \|X_0-x^\star\|^2 + 2\sup_{s\leq t}A_{s\wedge\tau_n} + 2\sup_{s\leq t}\mathscr{M}_{s\wedge\tau_n}\\
            &\leq \|X_0-x^\star\|^2 + 2\int_0^\infty\sigma_\infty(s)^2 ds + 2\sup_{s\leq t}\mathscr{M}_{s\wedge\tau_n}\\
        \end{aligned}
    \end{equation*}
    Taking expectation, we get
    \begin{equation*}
            \mathbb{E}(\sup_{s\leq t}\|X_{s\wedge \tau_n}-x^\star\|^2) \leq \mathbb{E}(\|X_0-x^\star\|^2) + 2\mathbb{E}\left(\int_0^\infty\sigma_\infty(s)^2 ds\right) + 2\mathbb{E}(\sup_{s\leq t}\mathscr{M}_{s\wedge\tau_n})
    \end{equation*}
  By Burkholder-Davis-Gundy inequality (see, e.g., \cite[Theorem 3.28]{MR1121940}), there exists a (universal)  constant $\mathcal{K}>0$ 
    \begin{equation}\label{eqBDG}
        \begin{aligned}
            \mathbb{E}(\sup_{s\leq t}\mathscr{M}_{s\wedge\tau_n}) &\leq \mathcal{K}\mathbb{E}(\|\sigma_\infty\|_2\cdot \sup_{s\leq t}\|X_{s\wedge\tau_n}-x^\star\|)\\
            &\leq \mathbb{E}\left(\mathcal{K}\|\sigma_\infty\|_2\cdot \left(\frac{1}{4 \mathcal{K}\|\sigma_\infty\|_2}\sup_{s\leq t}\|X_{s\wedge\tau_n}-x^\star\|^2 + \mathcal{K}\|\sigma_\infty\|_2\right)\right)\\
            &=\frac{1}{4}\mathbb{E}\left(\sup_{s\leq t}\|X_{s\wedge\tau_n}-x^\star\|^2\right) + \mathcal{K}^2\mathbb{E}\left(\|\sigma_\infty\|_2^2\right).
        \end{aligned}
    \end{equation}
    It follows that 
    \begin{equation*}
        \mathbb{E}(\sup_{s\leq t}\|X_{s\wedge \tau_n}-x^\star\|^2) \leq 2\mathbb{E}(\|X_0-x^\star\|^2) +4( \mathcal{K}^2+ 1)\mathbb{E}(\|\sigma_\infty\|_2^2).
    \end{equation*}
    By Fatou's lemma, taking $n\to\infty$ and then $t\to \infty$ we conclude that
    \begin{equation*}
        \mathbb{E}(\sup_{s\geq 0}\|X_s-x^\star\|^2)<\infty.
    \end{equation*}
    Then, \emph{(b)} holds. Moreover, observe that by taking expectation in \eqref{eqn-ito11}, we have \emph{(a)} holds since the expectation of the Brownian integral is zero since it is a martingale due to the integrability just proved. 
    
    We have proved that for all $x^\star\in \mathcal{S}$, $\mathbb{P}$-a.s. $\lim_{t\to \infty}\|X_s-x^\star\|$ exists. Our goal now is to prove a stronger result: $\mathbb{P}$-a.s. for all $x^\star\in \mathcal{S}$, $\displaystyle\lim_{t\to \infty}\|X_t-x^\star\|$ exists. To achieve this, we follow the argumentation of \cite{MR3361444}. Consider a countable dense subset of $\mathcal{S}$, denoted by $\mathcal{S}'$. For every $x^\star\in \mathcal{S}'$ we take $\Omega_{x^\star}\in \mathcal{F}$ such that $\mathbb{P}(\Omega_{x^\star}) = 1$ and for all $\omega \in \Omega_{x^\star}$, $\displaystyle\lim_{t\to \infty}\|X_t(\omega)-x^\star\|$ exists. Define $\Omega' := \bigcap_{x^\star\in \mathcal{S}'}\Omega_{x^\star}$, given that $\mathcal{S}'$ is countable we have $\Omega'\in \mathcal{F}$ and $\mathbb{P}(\Omega') = 1$. As a consequence, if $x^\star\in \mathcal{S}$, there is a sequence $(x_i^\ast)\subset\mathcal{S}'$ such that $x_i^\ast\to x^\star$, and we observe that for $\omega\in \Omega'$
    \begin{equation*}
    	\forall i\in \N, \forall t\geq 0 : \left|\|X_t(\omega)-x^\star\|-\|X_t(\omega)-x_i^\ast\|\right|\leq \|x_i^\ast-x^\star\|.
    \end{equation*}
    By using the last inequality we have for all $i\in \N$
    \begin{equation*}
    	\begin{aligned}
    		\limsup_{t\to \infty}\|X_t(\omega)-x^\star\|&\leq \lim_{t\to \infty}\|X_t(\omega)-x_i^\ast\|+\|x_i^\ast-x^\star\|\\
    		&\leq \liminf_{t\to \infty}\|X_t(\omega)-x^\star\|+2\|x_i^\ast-x^\star\|.
    	\end{aligned}
    \end{equation*}
    Then, taking $i\to \infty$ we conclude that $\displaystyle\lim_{t\to \infty}\|X_t(\omega)-x^\star\|$ exists. Furthermore, by \emph{(b)} and dominated convergence theorem, we obtain directly that \eqref{conv-L2-zeroe} holds, then \emph{(c)} is proved.
    
    Now, we are going to prove that $\Aver{X}{t} $ converges $\mathbb{P}$-a.s. to a random element of $\mathcal{S}$. First, define  $N_t := \int_0^t P^\ast\sigma(\tau,X_\tau)dB_\tau$, observe that each coordinate of $N$ is a martingale bounded in $L^2(\Omega)$ since $\sigma_\infty\in L^2(\R_+)$, then by the martingale convergence theorem (see, e.g., \cite[Theorem 1.3.5]{MR2380366}), there is a random variable $N_\infty\colon \Omega\to \R^d$, such that $\mathbb{P}$-a.s. $N_t\to N_\infty$ and $\displaystyle\lim_{t\to \infty}\mathbb{E}(\|N_t-N_\infty\|^2)=0$.

    Take $(u,v)\in \gph A$ and any $x^\star\in \mathcal{S}$. Consider the set $\Omega'\in \mathcal{F}$ such that for all $\omega\in \Omega'$: $\displaystyle\lim_{t\to \infty} N_t(\omega) = N_\infty(\omega)$, $\int_0^\infty\langle X_s(\omega)-x^\star,dY_s(\omega)\rangle<\infty$, $\displaystyle\lim_{t\to \infty}\|X_t(\omega)-x^\star\|$ exists and $\langle X_s(\omega)-u,dY_s(\omega)-vds\rangle$ is a positive measure. We know that $\mathbb{P}(\Omega') =1$. Let us take $\omega \in \Omega'$, then have for all $t\geq 0$
    \begin{equation*}
    	\int_0^t \langle u-X_s(\omega),v \rangle ds-\int_0^t\langle  u-X_s(\omega),dY_s(\omega)\rangle\geq 0.
    \end{equation*}
     Observe that for all $t> 0$
    \begin{equation*}
        \begin{aligned}
            &\left\langle u-\Aver{X}{t}(\omega) ,v\right\rangle
            = \frac{1}{t}\int_0^t\langle u-X_s(\omega),v\rangle ds
        \geq  \frac{1}{t}\int_0^t\langle u-X_s(\omega),dY_s(\omega)\rangle\\
        &= \frac{1}{t}\int_0^t\langle x^\star-X_s(\omega),dY_s(\omega)\rangle + \frac{1}{t}\langle u-x^\star,Y_t(\omega)+X_0(\omega)\rangle\\
        & = \frac{1}{t}\int_0^t\langle x^\star-X_s(\omega),dY_s(\omega)\rangle  + \frac{1}{t}\langle u-x^\star,-X_t(\omega) + N_t(\omega) +X_0(\omega)\rangle.
            \end{aligned}
    \end{equation*}
    Since $\displaystyle\lim_{t\to \infty}\|X_t(\omega)-x^\star\|$ exists, $N_t(\omega)\to N_\infty(\omega)$ and $\int_0^\infty\langle X_s(\omega)-x^\star,dY_s(\omega)\rangle<\infty$ we obtain that 
    \begin{equation}\label{eqn-limit-erg111}
    	 \lim_{t\to \infty} \frac{1}{t}\int_0^t\langle x^\star-X_s(\omega),dY_s(\omega)\rangle + \frac{1}{t}\langle u-x^\star,-X_t(\omega) + N_t(\omega) +X_0(\omega)\rangle = 0.
    \end{equation}
    Suppose that $(t_n)$ is a sequence such that $t_n\to \infty$ and $\Aver{X}{t_n}(\omega)\to \bar x$, then the above inequality implies that 
    \begin{equation*}
        \langle v,u-\bar x \rangle\geq 0.
    \end{equation*}
    Since $A$ is maximal monotone, it follows that $\bar x\in \mathcal{S}$. Hence, $\mathbb{P}$-a.s. every cluster point of $(\Aver{X}{t})_{t\in \R_+}$ belongs to $\mathcal{S}$. By \emph{(c)} and Lemma \ref{opial-prop}  we conclude that $\mathbb{P}$-a.s. $\Aver{X}{t} $ converges to an element of $\mathcal{S}$. We denote $\displaystyle X^\star:=\lim_{t\to\infty}\Aver{X}{t} $. By using \emph{(b)}, we have $\sup_{t\geq 0}\|X_t\|^2\in L^1(\Omega)$, then by Fatou's Lemma
    \begin{equation*}
        \mathbb{E}(\|X^\star\|^2)\leq \lim_{t\to \infty}\mathbb{E}(\|\Aver{X}{t} \|^2)\leq \mathbb{E}(\sup_{s\geq 0}\|X_s\|^2)<\infty,
    \end{equation*}
    thus $X^\star\in L^2(\Omega;\R^d)$. Then, the proof of \emph{(d)} is done.
    \end{proof}
    
     \begin{remark}\label{F0-measurable-remark}
    	It is worth noting that assertions \emph{(a)} of Theorem \ref{main-thm-max-mon} can be stated when $x^\star\in \mathcal{S}$ is replaced by a random variable $X^\star\colon \Omega\to \mathcal{S}$ which is $\mathcal{F}_s$-measurable and $X^\star\in L^2(\Omega;\R^d)$.
    \end{remark}

In the following result, we formally prove that the solution $(X_t)$ of the SDI~\eqref{MainSDI} converges to a measurable function taking values in $\mathcal{S}$, provided that the solution of the deterministic differential inclusion 
\begin{align}\label{detDS}
    \dot{x}(t) \in -A(x(t)), \qquad x(0) = x_0,
\end{align}
converges to an equilibrium point of $A$ for any initial condition $x_0 \in \cl(\dom A)$. 

It is worth recalling that the existence of solutions for this class of differential inclusions has a long history (see, e.g., \cite{MR755330} and the references therein). Moreover, it is well known that such solutions do not necessarily converge to an equilibrium of $A$. Nevertheless, several authors have investigated sufficient conditions ensuring that the solution of the differential inclusion converges to an equilibrium point (see, e.g., \cite{MR377609,MR477932}). 

% \textcolor{red}{Perhaps the most important case for which global convergence applies is the subgradient evolution problem $A=\partial f$, for a convex lower-semi continuous function $f$. In this case, we have a chain rule 
% $$
% \frac{d}{dt}(f\circ x)(t)=\langle g(t),\dot{x}(t)\rangle\quad \forall g(t)\in\partial f(x(t)),\text{a.e. }t\in(0,T).
% $$
% By definition of a solution, it follows 
% $$
% \frac{d}{dt}(f\circ x)(t)=-||\dot{x}(t)||^{2},\text{a.e. }t\in(0,T).
% $$
% This shows that $f$ is a Lyapunov function for the trajectory $x$, and 
% $$
% \infty>f(x(0))-f(x(t))=\int_{0}^{t}||\dot{x}(s)||^{2}ds. 
% $$
% A convex lower-semi continuous function is bounded from below and hence $x(\cdot)\in W^{1,2}([0,T];\R^{d})$.
% }

\begin{theorem}\label{general_convergence_SDI}
   Under the assumptions of Theorem~\ref{main-thm-max-mon}, suppose that $A$ satisfies that every trajectory of the deterministic system~\eqref{detDS} starting from any $x_0 \in \operatorname{cl}(\operatorname{dom} A)$ converges to an equilibrium. Then there exists a random variable $X^\star \colon \Omega \to \mathcal{S}$ such that 
\begin{equation*}
        \mathbb{P}\text{-a.s. } \lim_{t\to \infty}X_t = X^\star \text{ and } \lim_{t\to \infty}\mathbb{E}(\|X_t-X^\star\|^2) = 0.
    \end{equation*}
\end{theorem}
\begin{proof}
For every $T>0$, consider $x_T\colon [T,\infty[\to \R^d$, the unique absolutely continuous function satisfying \eqref{detDS}, with initial condition $x_T(T) = X_T$. By our assumption, we have $\displaystyle x_T^\ast := \lim_{t\to \infty} x_T(t)\in \mathcal{S}$. Moreover, if $\bar x\in \mathcal{S}$, we have $\|x_T(t)-\bar x\|\leq \|X_T-\bar x\|, \forall t\geq T$. Since $X_T\in L^2_{\mathbb{P}}(\Omega;\R^d)$, it follows that $x_T(t)\in L^2_{\mathbb{P}}(\Omega;\R^d)$ and $x_T^\ast\in L^2_{\mathbb{P}}(\Omega;\R^d)$. By using \cite[Proposition 6.17]{MR3308895}, we have
\begin{equation*}
    \int_a^b \langle X_t - x_{T}(t) , dY_t\rangle \geq \int_a^b \langle X_t - x_{T}(t) ,-x'_{T}(t)\rangle dt, \forall a\geq b\geq T.
\end{equation*}
The above conclusion holds $\mathbb{P}$-a.s.. Moreover, by using Itô's Lemma we have for  $t\geq T$
\begin{equation*}
    \begin{aligned}
    \frac{1}{2}\|X_t-x_T(t)\|^2 = & \    \int_T^t\langle X_s-x_T(s), -dY_s\rangle + \int_T^t\langle X_s-x_T(s), -x_T'(s)\rangle ds\\
    & + \int_T^t\langle X_s-x_T(s), P^\ast\sigma(s,X_s)dB_s\rangle\\
    &+ \frac{1}{2}\int_T^t\|P^\ast\sigma(s,X_s)\|^2 ds\\
    \leq & \ \int_T^t\langle X_s-x_T(s), P^\ast\sigma(s,X_s)dB_s\rangle + \frac{1}{2}\int_T^t\sigma_\infty(s)^2 ds.
    \end{aligned}
\end{equation*}
Taking expectation, we conclude 
\begin{equation*}
    \forall t\geq T : \mathbb{E}(\|X_t-x_T(t)\|^2)\leq \int_T^t\mathbb{E}(\sigma_\infty(s)^2)ds.
\end{equation*}
Since the expectation of martingale part of the right hand side is zero. On the other hand, observe that if $t\geq T$
\begin{equation*}
    d(X_t;\mathcal{S})^2\leq \|X_t-x^\ast_T\|^2\leq 2\|X_t-x_T(t)\|^2 + 2\|x_T(t)-x_T^\ast\|^2,
\end{equation*}
where $d(X_t;\mathcal{S})$ denotes the distance form $X_t$ to $\mathcal{S}$. By taking expectation, we have
\begin{equation*}
    \begin{aligned}
        \mathbb{E}(d(X_t;\mathcal{S})^2)&\leq 2\mathbb{E}(\|X_t-x_T(t)\|^2) + 2\mathbb{E}(\|x_T(t)-x_T^\ast\|^2)\\
        &\leq 2\int_T^\infty \mathbb{E}(\sigma_\infty(s)^2) ds + 2\mathbb{E}(\|x_T(t)-x_T^\ast\|^2). 
    \end{aligned}
\end{equation*}
Since $\forall t\geq T : \|x_T(t)-x_T^\ast\|^2\leq \|X_T-x_T^\ast\|^2\in L^1(\Omega)$ and $\displaystyle\lim_{t\to \infty}\|x_T(t)-x_T^\ast\|^2 = 0$, by using dominated convergence theorem we have $\displaystyle\lim_{t\to \infty}\mathbb{E}(\|x_T(t)-x_T^\ast\|^2) = 0$. Therefore,
\begin{equation*}
    \limsup_{t\to \infty}\mathbb{E}(d(X_t;\mathcal{S})^2)\leq 2\int_T^\infty \mathbb{E}(\sigma_\infty(s)^2) ds.
\end{equation*}
 Since $T>0$ was arbitrary and $\sigma_\infty\in L^2(\R_+)$, we conclude that
\begin{equation*}
    \lim_{t\to \infty}\mathbb{E}(d(X_t;\mathcal{S})^2) = 0.
\end{equation*}
It follows that there is $(t_n)\to \infty$ such that $\mathbb{P}$-a.s. $\displaystyle\lim_{n\to \infty}d(X_{t_n};\mathcal{S})^2 = 0$. Take the set $\Omega'\subset\Omega$ such that for all $\omega\in \Omega'$, $\displaystyle\lim_{n\to \infty}d(X_{t_n}(\omega);\mathcal{S})^2 = 0$ and for all $\displaystyle x^\star\in \mathcal{S} : \lim_{t\to \infty}\|X_t(\omega)-x^\star\| $ exists, we know that $\mathbb{P}(\Omega') = 1$. Fix $\omega\in \Omega'$, since  $(X_{t_n}(\omega))$ is bounded, we have up to a subsequence $X_{t_n}(\omega)\to x^\star$ and given that $\displaystyle \lim_{n\to \infty}d(X_{t_n}(\omega);\mathcal{S})=0$, we have $x^\star\in \mathcal{S}$. Since $\displaystyle\lim_{t\to \infty}\|X_t(\omega)-x^\star\|$ exists, we conclude that $\displaystyle\lim_{t\to \infty}\|X_t(\omega)-x^\star\| = 0$. Since $\omega\in \Omega'$ was arbitrary, we conclude that $\mathbb{P}$-a.s. $\displaystyle X^\star :=\lim_{t\to \infty} X_t\in\mathcal{S}$. Finally, by Fatou's lemma we have $\mathbb{E}(\|X^\star\|^2)\leq \mathbb{E}(\sup_{t\geq 0}\|X_t\|^2)$ and by Theorem \ref{main-thm-max-mon} we conclude that $X^\star\in L^2_{\mathbb{P}}(\Omega;\R^d)$ and by dominated convergence theorem we conclude that $\displaystyle\lim_{t\to \infty}\mathbb{E}(\|X_t-X^\star\|^2) = 0$.
\end{proof}

Perhaps the most important case in which global convergence of solutions to \eqref{detDS} applies is the
subgradient evolution problem, that is,  $A = \partial f$, where $f$ is a proper, convex, and lower semicontinuous
function; see, for instance, \cite{MR755330}. Nonetheless, an equally important and more general setting
is provided by the notion of \emph{demipositive} maximal monotone operators, which in particular covers
the case of subdifferentials of convex functions, cocoercive operators, strongly maximal monotone operators among others  (see, e.g., \cite[Proposition 6.2]{MR2731260}).

We recall that a maximal monotone operator $A$ is said to be \emph{demipositive} if there exists
$w \in \mathcal{S}$ such that, for every sequence $(u_n) \subset \dom A$ converging to $u$, and every
bounded sequence $(v_n)$ with $v_n \in A(u_n)$, the implication
\begin{equation*}
    \langle v_n, u_n - w \rangle \to 0 
    \quad \Longrightarrow \quad 
    u \in \mathcal{S}
\end{equation*}
holds. This notion was introduced in \cite{MR377609} to establish the convergence of solutions of
the deterministic system~\eqref{detDS} to equilibrium points of the maximal monotone operator.

\begin{corollary}\label{cor-theo-conv-general}
    Under the assumptions of Theorem~\ref{main-thm-max-mon} suppose that  $A$ is demipositive. Then, there is some random variable $X^\star\in L^2_{\mathbb{P}}(\Omega;\R^d)$ such that 
    \begin{equation*}
        \mathbb{P}\text{-a.s. } \lim_{t\to \infty}X_t = X^\star \text{ and } \lim_{t\to \infty}\mathbb{E}(\|X_t-X^\star\|^2) = 0.
    \end{equation*}
\end{corollary}
\begin{proof}
    As a consequence of \cite[Corollary III.3.1]{MR348562}, the system \eqref{detDS} has an unique solution for all $x_0\in \cl(\dom A)$. The result follows from \cite[Theorem 6.3]{MR2731260} and Theorem \ref{general_convergence_SDI}.
\end{proof}
\begin{remark}
    The above corollary ensures that when $A$ is either strongly maximal monotone or cocoercive, the conclusion of Theorem \ref{general_convergence_SDI} holds since in both cases $A$ is demipositive. Particularly, Corollary~\ref{cor-theo-conv-general} generalizes the results obtained in \cite[Theorem 5.1 \& 5.2]{MR4989859}, where only cocoercive operators where considered.
\end{remark}
The next results provides an error bound in terms of a suitably defined gap function associated with the maximally monotone operator. Consider the Br\'{e}zis-Haraux function 
$$
H_{A}(x,y)=\sup\{\langle x-u,v-y\rangle :(u,v)\in\gph(A)\}.
$$
It is well known \cite{MR3729487} that $H_{A}$ can be interpreted as an exterior penalty for $\gph(A)$, in the sense that $H_{A}(x,u)=0$ if and only if $(x,u)\in \gph(A)$. Moreover, it holds true that 
$$
H_{A}(x,y)\geq||x-(\text{Id}+A)^{-1}(x+y)||^{2}\geq 0 \qquad\forall (x,y)\in\R^{d}\times\R^{d}. 
$$
Given these properties, a standard merit function for solving the monotone inclusion problem is provided in terms of the gap function $G_A(x):=H_{A}(x,0)$, or its localized version
\begin{equation}
\label{eq:gap}
G_A(x\vert K)=\sup\left\{ \langle x-u,v\rangle : u\in\dom(A)\cap K,\; v\in Au \right\},
\end{equation}
 where $K\subset\R^{d}$ is a given compact set. In \cite{MR3575644}, the authors show that $\Gamma(x)$ is the minimal translation invariant gap function for the problem of finding an element of $A^{-1}(0)$. The localized version defined in \eqref{eq:gap} is used to obtain a well-defined merit function for equilibrium problems having an unbounded domain (see e.g. \cite{MR2295146}). 

    \begin{proposition}
    Under the assumptions of Theorem~\ref{main-thm-max-mon}, consider the process $X$ solving \eqref{MainSDI}. Then, for all compact set $K\subset \R^d$, we have  that
            \begin{equation*}
            \mathbb{E}[G_A(\Aver{X}{t}\vert K)]\leq \mathcal{O}(1/t).
                %\mathbb{E}\left(\sup\{ \langle v,\Aver{X}{t}-u\rangle :(u,v)\in \gph A\cap K\times \R^d\}\right)\leq \mathcal{O}\left(\frac{1}{t}\right).
            \end{equation*}
In addition, assume that $A$ is $\rho$-strongly maximal monotone with $\mathcal{S} = \{x^\star\}$, then
    		\begin{equation}\label{eqn-strong-mon-rate}
    			\mathbb{E}(\|X_t-x^\star\|^2)\leq e^{-2\rho t}\mathbb{E}(\|X_0-x^\star\|^2) + \int_0^te^{-2\rho(t-s)}\mathbb{E}(\sigma_\infty(s)^2)ds.
    		\end{equation}
    \end{proposition}
    \begin{proof}
   Take a compact set $K\subset \R^d$, note that in the proof of Theorem \ref{MainSDI} it is proved that for all $(u,v)\in \gph A$ and $x^\star\in \mathcal{S}$
        \begin{equation*}
            \begin{aligned}
            \left\langle v,\Aver{X}{t}-u\right\rangle
             \leq \frac{1}{t}\int_0^t\langle X_s-x^\star,dY_s\rangle  + \frac{1}{t}\langle u-x^\star,X_t-X_0- N_t\rangle.
            \end{aligned}
        \end{equation*}
        where $N_t := \int_0^t \sigma(t,X_t)dB_t$. Then, taking the supremum and then expectation it yields
    \begin{equation*}
        \begin{aligned}
             \mathbb{E}[G_A(\Aver{X}{t}\vert K)] 
        \leq & \  \frac{1}{t}\mathbb{E}\left(  \int_0^\infty\langle X_s-x^\star,dY_s\rangle\right)\\
        &+ \frac{1}{t}(\sup_{w\in K}\|w\|+ \|x^\star\|)(2\mathbb{E}(\sup_{s\geq 0}\|X_s\|^2)^{1/2}  +\sup_{s\geq 0}\mathbb{E}(\|N_s\|^2)^{1/2}),
        \end{aligned}
    \end{equation*}
   which gets the desired error bound.
    
    \noindent Now, define the function $\psi(t) := \mathbb{E}(\|X_t-x^\star\|^2)$. By Fatou's Lemma, $\psi$ is lsc. According to Theorem \ref{main-thm-max-mon} we have that for all $t,s\in\R_+$ with $t\geq s$
    	\begin{equation*}
    		\begin{aligned}
    			\psi(t)&\leq \psi(s) -2\mathbb{E}\left(\int_s^t\langle X_\tau-x^\star,dY_\tau\rangle\right)+\int_s^t\mathbb{E}(\sigma_\infty(\tau)^2)d\tau\\
    			&\leq \psi(s) -2\rho\mathbb{E}\left(\int_s^t\|X_\tau-x^\star\|^2 d\tau\right)+\int_s^t\mathbb{E}(\sigma_\infty(\tau)^2)d\tau\\
    			&= \psi(s) -2\rho\int_s^t\psi(\tau)d\tau+\int_s^t\mathbb{E}(\sigma_\infty(\tau)^2)d\tau.
    		\end{aligned}
    	\end{equation*}
    	By using Lemma \ref{lemma-comparison} we have that
    	\begin{equation*}
    		\psi(t)\leq e^{-2\rho t}\psi(0) + \int_0^te^{-2\rho(t-s)}\mathbb{E}(\sigma_\infty(s)^2)ds,
    	\end{equation*}
        which proves \eqref{eqn-strong-mon-rate}.
    \end{proof}
    \section{Stochastic Subdifferential Inclusion}\label{Section5}

In this section, we consider a proper, convex, and lower semicontinuous function 
$\varphi \colon \mathbb{R}^d \to \mathbb{R} \cup \{\infty\}$, and study the asymptotic behavior of the solution to the SDI \eqref{MainSDI} in the particular case $A = \partial \varphi$. This analysis builds on the previous results and in this setting we can derive additional convergence estimates, including estimates on the values of $\varphi$ along the trajectory. For related results in the smooth convex setting, we refer the reader to \cite{MR4989859, MR3747470}.

First, let us recall the following lemma presented in  \cite[Proposition 6.36]{MR3308895}.
    \begin{proposition}\label{prop-ineq-convex}
    	For all continuous and adapted process $\alpha\colon \R_+\to \R^d$, it holds $\mathbb{P}$-a.s. 
    	\begin{equation*}
    		\forall a,b\in \R_+, a\leq b : \int_a^b \varphi(\alpha_t)dt\geq \int_a^b\varphi(X_t)dt + \int_a^b\langle \alpha_t-X_t,dY_t\rangle.
    	\end{equation*}
    \end{proposition}
We now state the main result on convergence in the setting of stochastic subdifferential inclusions.

    %From Proposition \ref{prop-ineq-convex} and Theorem \ref{main-thm-max-mon} we have
    \begin{theorem}\label{Teo:main:convexfunction}
     Under the assumptions of Theorem~\ref{main-thm-max-mon}, suppose that $A=\partial \varphi$ for given  proper, convex, and lower semicontinuous function 
$\varphi \colon \mathbb{R}^d \to \mathbb{R} \cup \{\infty\}$. Then,
    	\begin{enumerate}[label = (\alph*)]
    		\item $\mathbb{E}\left(\int_0^\infty (\varphi(X_s)-\min \varphi )ds\right) < \infty$. Consequently, $\mathbb{P}$-a.s. $$\int_0^\infty (\varphi(X_t)-\min \varphi) dt<\infty \text{ and }\displaystyle\liminf_{t\to \infty} \varphi(X_t) = \min \varphi.$$
    		\item There is a random variable $X^\star\in L^2_{\mathbb{P}}(\Omega;\R^d)$ such that $\mathbb{P}$-a.s. $\displaystyle\lim_{t\to \infty}X_t = X^\star$, $\mathbb{P}(X^\star\in \mathcal{S}) = 1$ and $\displaystyle\lim_{t\to \infty}\mathbb{E}(\|X_t-X^\star\|^2) = 0$.
            \item Suppose that
            \begin{equation}\label{cond-conv-varphi}
                \displaystyle\limsup_{t\to \infty}\mathbb{E}(\|\partial^0\varphi(X_t)\|^2)<\infty,
            \end{equation} then $\displaystyle\lim_{t\to \infty} \mathbb{E}(\varphi(X_t)) = \min \varphi$.
            \item It holds $\mathbb{P}$-a.s., $\displaystyle\lim_{t\to \infty}\varphi(\Aver{X}{t}) = \min \varphi$. Furthermore,
              $$  \mathbb{E}(\varphi(\Aver{X}{t} )-\min \varphi)\leq \frac{1}{2t}\mathbb{E}(d(X_0;\mathcal{S})^2) + \frac{1}{2t}\int_0^t\mathbb{E}(\sigma_\infty(s)^2) ds.$$
              \item Suppose that $\varphi$ is $\mu$-strongly convex, then $\mathcal{S} = \{x^\star\}$ and we have \begin{equation}\label{ineq-aver-varphi1}
        \mathbb{E}(\|\Aver{X}{t} -x^\star\|^2)\leq \frac{1}{\mu t}\mathbb{E}(\| X_0 - x^\star\|) + \frac{1}{\mu t}\int_0^t\mathbb{E}(\sigma_\infty(s)^2) ds.
    \end{equation}
    	\end{enumerate}
    
    \end{theorem}
    \begin{proof}
    In this setting, the set of equilibrium points is given by $\mathcal{S} = \operatorname{argmin} \varphi$. Let us take $\widehat{X} := \proj_{\mathcal{S}}(X_0)$. Observe that $\widehat{X}\in L^2_{\mathbb{P}}(\Omega;\R^d)$ and it is $\mathcal{F}_0$-measurable, by using Proposition \ref{prop-ineq-convex} we have $\mathbb{P}$-a.s.
	\begin{equation*}\label{convexity-Xs}
		\int_0^t(\varphi(X_s)-\min\varphi )ds\leq \int_0^t\langle X_s-\widehat{X},dY_s\rangle.
	\end{equation*}
	Taking expectation, and by using Theorem \ref{main-thm-max-mon} (see Remark \ref{F0-measurable-remark}) we have
	\begin{equation}\label{eqn-varphi-exp1}
		\begin{aligned}
			\mathbb{E}\left(\int_0^t(\varphi(X_s)-\min\varphi )ds\right)\leq \mathbb{E}(d(X_0;\mathcal{S})^2)+\int_0^t\mathbb{E}(\sigma_\infty(s)^2) ds.
		\end{aligned}
	\end{equation}
	By Beppo Levi's convergence Theorem we have that
	\begin{equation*}
		\mathbb{E}\left( \int_0^\infty(\varphi(X_s)-\min\varphi )ds \right) \leq \mathbb{E}(d(X_0;\mathcal{S})^2)+\int_0^\infty\mathbb{E}(\sigma_\infty(s)^2) ds<\infty.
	\end{equation*}
	In particular, we have that $\mathbb{P}$-a.s. $\int_0^\infty(\varphi(X_s)-\min\varphi )ds<\infty$, thus by using Lemma \ref{lem-convergence-zero} we have $\mathbb{P}$-a.s. $\displaystyle\liminf_{t\to \infty}\varphi(X_t)=\min\varphi$, then \emph{(a)} is proved.\\
    To check \emph{(b)}, we note that $\partial \varphi$ is demipositive (see, e.g., \cite[Proposition 6.2]{MR2731260}), then by Corollary \ref{cor-theo-conv-general}, \emph{(b)} holds.
    
    To prove \emph{(c)}, we simply note that the condition $\displaystyle\limsup_{t\to \infty}\mathbb{E}(\|\partial^0\varphi(X_t)\|^2)<\infty$ implies that there is $T>0$ such that for a.e. $t\geq T$, $\partial^0\varphi(X_t)$ is well-defined. Hence, for a.e. $t\geq T$
    \begin{equation*}
        \varphi(X_t)-\varphi(x^\star)\leq \langle \partial^0\varphi(X_t), X_t-x^\star\rangle\leq \|\partial^0\varphi(X_t)\|\|X_t-x^\star\|
    \end{equation*}
    By taking expectation and using Hölder inequality we have
    \begin{equation*}
        \mathbb{E}(\varphi(X_t)-\min \varphi)\leq \mathbb{E}(\|\partial^0\varphi(X_t)\|^2)^{1/2}\mathbb{E}(\|X_t-x^\star\|^2)^{1/2}.
    \end{equation*}
    Thus, \emph{(c)} holds by taking $t\to \infty$ since $\displaystyle\lim_{t\to \infty}\mathbb{E}(\|X_t-x^\star\|^2) = 0$.
    On the other hand, by Lemma \ref{jensen-convex}, for all $t\in \R_+$
    \begin{equation*}
        \varphi(\Aver{X}{t})-\min\varphi\leq \frac{1}{t}\int_0^t(\varphi(X_s)-\min \varphi) ds\leq \frac{1}{t}\int_0^\infty(\varphi(X_s)-\min \varphi) ds.
    \end{equation*}
    It follows that $\mathbb{P}$-a.s. $\displaystyle\lim_{t\to \infty} \varphi(\Aver{X}{t})=\min \varphi$. Moreover, dividing by $t$ the inequality \eqref{eqn-varphi-exp1} yields
    \begin{equation*}
        \mathbb{E}\left(\frac{1}{t}\int_0^t \varphi(X_s)ds-\min \varphi\right)\leq \frac{1}{2t}\mathbb{E}(d(X_0;\mathcal{S})^2) + \frac{1}{2t}\int_0^t\mathbb{E}(\sigma_\infty(s)^2) ds
    \end{equation*}
    and by Lemma \ref{jensen-convex}, we conclude that 
    \begin{equation}\label{ineq-values-aver1}
        \mathbb{E}(\varphi(\Aver{X}{t} )-\min \varphi)\leq \frac{1}{2t}\mathbb{E}(d(X_0;\mathcal{S})^2) + \frac{1}{2t}\int_0^t\mathbb{E}(\sigma_\infty(s)^2) ds.
    \end{equation}
    Hence, \emph{(d)} holds. Finally, suppose that $\varphi$ is $\mu$-strongly convex, then $\mathcal{S} = \{x^\star\}$ and 
    \begin{equation*}
        \frac{\mu}{2}\|\Aver{X}{t} -x^\star\|^2 + \varphi(x^\star)\leq \varphi(\Aver{X}{t} ).
    \end{equation*}
    Taking expectation and using inequality \eqref{ineq-values-aver1} we obtain directly \eqref{ineq-aver-varphi1}, thus \emph{(e)} is done.
	\end{proof}
    \begin{remark}
        A sufficient condition to ensure \eqref{cond-conv-varphi} is to assume that the function $\varphi$ satisfies $\dom \partial\varphi$ is closed and there is $c>0$ such that $\|\partial^0\varphi(x)\|\leq c(\|x\|+1), \forall x\in \dom \partial\varphi$. 
    \end{remark}
We close this section with establishing a concentration bound of the ergodic averaged trajectory in terms of the objective function value gap 
$$
\Delta\varphi(x):=\varphi(x)-\min\varphi.
$$
Let $x^{\star}$ be a solution, and define $E_{x^{\star}}(t):=\frac{1}{2}||X(t)-x^{\star}||^{2}$. Itô's formula yields 
$$
\frac{1}{t}\int_{0}^{t}\Delta\varphi(X(s))ds \leq \frac{E_{x^{\star}}(0)}{t}+\frac{1}{t}\int_{0}^{t}\langle X_{s}-x^{\star},P^\ast\sigma(s,X_{s})dB_{s}\rangle+\frac{1}{2t}\int_{0}^{t}\sigma_{\infty}(s)^{2}ds. 
$$
Applying Jensen's inequality, this yields 
$$
\Delta\varphi(\Aver{X}{t})\leq \frac{E_{x^{\star}}(0)}{t}+\frac{1}{t}\int_{0}^{t}\langle X_{s}-x^{\star},P^\ast\sigma(s,X_{s})dB_{s}\rangle+\frac{1}{2t}\int_{0}^{t}\sigma_{\infty}(s)^{2}ds. 
$$
To control the martingale part in the middle of the above expression, we introduce Wiener process $(W_{t})_{t\geq 0}$, defined by 
$$
dW_{t}=-P^\ast\sigma(t,X_{t})dB_{t},\quad W_{0}=X_{0}.
$$
Applying Itô's formula to the process $t\mapsto \pi(t)=\frac{1}{2}||W_{t}-x^{\star}||^{2}$, we obtain 
$$
\int_{0}^{t}\langle W_{s}-x^{\star},P^\ast\sigma(s,X_{s})dB_{s}\rangle \leq E_{x^{\star}}(0)+\int_{0}^{t}\sigma^{2}_{\infty}(s)ds.
$$
This allows us to decompose the continuous martingale term as follows:
\begin{align*}
\int_{0}^{t}\langle X_{s}-x^{\star},P^\ast\sigma(s,X_{s})dB_{s}\rangle=&\int_{0}^{t}\langle X_{s}-W_{s},P^\ast\sigma(s,X_{s})dB_{s}\rangle\\
&+\int_{0}^{t}\langle W_{s}-x^{\star},P^\ast\sigma(s,X_{s})dB_{s}\rangle\\
\leq &E_{x^{\star}}(0)+\int_{0}^{t}\sigma^{2}_{\infty}(s)ds+U_{t}
\end{align*}
where $U_{t}:=\int_{0}^{t}\langle X_{s}-W_{s},P^\ast\sigma(s,X_{s})dB_{s}\rangle$. Hence, 
$$
\Delta\varphi(\Aver{X}{t})\leq\frac{1}{t}(S_{t}+U_{t}),\text{where }S_{t}:=2E_{x_{^{*}}}(0)+\int_{0}^{t}\sigma^{2}_{\infty}(s) ds.
$$
For $\epsilon>0$, we thus see 
$$
\mathbb{P}(\Delta\varphi(\Aver{X}{t})\geq\epsilon)\leq\mathbb{P}(U_{t}\geq t\epsilon-S(t)).
$$
This estimate is the basis for the following large-deviations bound.
\begin{theorem}
Under the assumptions of Theorem~\ref{Teo:main:convexfunction}, consider $x^{\star}$   an arbitrary solution. The, for all $\epsilon>0$ and $t>0$, we have
\begin{equation}
\mathbb{P}\left(\Delta\varphi(\bar{X}_{t})\geq Q_{0}(t)+\epsilon Q_{1}(t) \right)\leq \exp(-\epsilon^{2}/4),
\end{equation}
where 
\begin{align}
Q_{0}(t)&:= \frac{1}{t}\left(\int_{0}^{t}\sigma^{2}_{\infty}(s)ds+2 E_{x^{\star}}(0)\right),\\
Q_{1}(t)&:=  \frac{1}{t}\sqrt{\delta(t)},\text{ where }\delta(t):={\mathbb{E}}\left(\int_{0}^{t}\sigma^{2}_{\infty}(s)\cdot||W_{s}-X_{s}||^{2}ds\right).
\end{align}
\end{theorem}
\begin{proof}
Let be $\rho_{t}:=[U,U]_{t}$ the quadratic variation process of the continuous semimartingale $\{(W_{t},\mathcal{F}_{t});t\geq 0\}$. We have 
\begin{align*}
\rho_{t}&=\int_{0}^{t}\|(P^\ast\sigma(s,X_s))^{\top}(W_{s}-X_{s})\|^{2}ds \leq \int_{0}^{t}\|\sigma(s,X_s)\|^{2}\cdot\|W_{s}-X_{s}\|^{2}ds \\
&\leq \int_{0}^{t}\sigma^{2}_{\infty}(s)\|W_{s}-X_{s}\|^{2}ds=:\hat{\delta}(t),
\end{align*}
it follows that  $\delta(t) = \mathbb{E}(\hat{\delta}(t))$. We next bound the stochastic exponential $t\mapsto \exp(\theta U_{t}),\theta\in\mathbb{R}$ via the Cauchy-Schwarz-Bunyakowski inequality as 
\begin{align*}
\mathbb{E}\left[\exp(\theta U_{t})\right]&=\mathbb{E}\left[\exp(\theta U_{t}-b\rho_{t})\exp(b\rho_{t})\right]
\leq \sqrt{\mathbb{E}\left[\exp(2\theta U_{t}-2b\rho_{t})\right]}\cdot \sqrt{\mathbb{E}\left[\exp(2b\rho_{t})\right]}. 
\end{align*}
Setting $b=\theta^{2}$, the expression inside the first expected values is just the stochastic exponential of the process $t\mapsto 2\theta U_t$. Hence, $t\mapsto \exp(2\theta U_t)$ is a submartingale, and we conclude 
\begin{equation}\label{eq:Martbound}
\mathbb{E}\left[\exp(\theta U_{t})\right]\leq \sqrt{\mathbb{E}\left[\exp(2\theta^{2}\rho_{t})\right]}\leq {\exp(\theta^2\mathbb{E}(\rho_t))\leq \exp(\theta^{2}\delta(t)).} 
\end{equation}
where is used Jensen's inequality. It follows, for $a>0$,
\begin{align*}
\mathbb{P}[U_{t}\geq a]%&=\mathbb{P}\left[\exp(\theta U_{t})\geq \exp(\theta a)\right]\\
&\leq\exp(-\theta a)\cdot\mathbb{E}\left[\exp(\theta U_t)\right]
\leq \exp(-\theta a)\exp(\theta^{2}\delta(t))=\exp(-\theta a+\theta^{2}\delta(t)),
\end{align*}
where the first inequality is just Markov's inequality, and the second line uses \eqref{eq:Martbound}. Minimizing the expression on the right-hand side with respect to $\theta$, yields $\theta=\frac{a}{2\delta(t)}$, resulting in the upper bound
$$
\mathbb{P}[U_{t}\geq a]\leq \exp\left(-\frac{a^{2}}{4\delta(t)}\right)
$$
From these computations, we deduce that 
\[
\mathbb{P}(U_{t}\geq a)\leq \exp(-\epsilon^{2}/4),  \text{whenever $a\geq \epsilon\sqrt{\delta(t)}$.}
\] 
This leads us the inequalities  
\begin{align*}
\mathbb{P}\left(\Delta\varphi(\bar{X}_t)\geq Q_{0}(t)+\epsilon Q_{1}(t) \right)&\leq \mathbb{P}(U_{t}\geq tQ_{0}(t)+t\epsilon Q_{1}(t)-S_{t})\\
&=\mathbb{P}(U_{t}\geq \epsilon \sqrt{\delta(t)})\leq \exp(-\epsilon^{2}/4),
\end{align*}
which concludes the proof.
\end{proof}

\section{Tikhonov Regularization}\label{Section6}
The final section of this paper concerns a stochastic Tikhonov regularization of the SDI
\eqref{MainSDI}, given by \eqref{tikhonov-system}, where the parameter function
$\epsilon\colon \Omega\times\mathbb{R}_+\to \mathbb{R}_+\setminus\{0\}$ satisfies
$\lim_{t\to\infty}\epsilon(t)=0$.
The additional drift term $-\epsilon(t)X_t$ plays the role of a Tikhonov regularizer for
\eqref{MainSDI}. Such perturbations have attracted considerable interest, as they often
improve the asymptotic behavior of trajectories (see, e.g.,
\cite{MR2462703, MR1398330}).
The result below shows that every trajectory of \eqref{tikhonov-system} converges to the
minimum-norm zero of $A$, namely $\proj_{\mathcal{S}}(0)$, without imposing further
structural assumptions on $A$.

For every $\eta>0$, we define $x_\eta := J_{\frac{1}{\eta}A}(0)$. By \cite[Theorem 23.48]{MR3616647},
we know that $\displaystyle\lim_{\eta\to 0} x_\eta = \proj_{\mathcal{S}}(0) =: x^\star$. Moreover, it
can be shown that the mapping $\eta \mapsto x_\eta$ is locally Lipschitz on
$\mathbb{R}_+ \setminus \{0\}$ and that, almost everywhere,
\begin{equation}\label{bound-derivative-curve1}
    \left\|\frac{d}{d\eta} x_\eta\right\| \leq \frac{\|x^\star\|}{\eta}
\end{equation}
(see, e.g., \cite[p.~529]{MR1398330}).

Recently, \cite[Theorem 14]{MR4991707} established a related result in the setting
where $A = \nabla F + \partial g$ on a Hilbert space, with $F$ a convex function of class
$\mathcal{C}^{1,+}$ and $g$ a continuous convex function. Their analysis requires the curve
$t \mapsto x_{\epsilon(t)}$ to satisfy a certain compatibility condition with the parameter function
$\epsilon$. In contrast, we do not impose any assumption on the trajectory $t \mapsto x_{\epsilon(t)}$;
instead, we only assume the standard conditions on $\epsilon$, which allows us to recover and extend
the deterministic result of \cite[Proposition 5]{MR2462703} to the stochastic
setting in finite dimension. Formally, we need to compare the stochstic dynamic give by the solution of \eqref{tikhonov-system} to its deterministic counterpart. Formally, for the initial condition $x_0\in \dom A$, we take the absolutely continuous trajectory $x\colon \R_+\to \R^d$ which satisfies
\begin{equation}\label{eqn_tikhonov-deterministic}
    x'(t)\in -A(x(t)) - \epsilon(t)x(t), \ x(0) = x_0.
\end{equation}
On the parameter function $\epsilon$, we now assume that it does not depend on $\Omega$, and only satisfies 
\begin{theorem}\label{Main_resultTych}
  Let $A \colon \mathbb{R}^d \rightrightarrows \mathbb{R}^d$  be  a maximal monotone operator  with $\mathcal{S}\neq \emptyset$.  Consider $\sigma \colon \Omega \times  \R_+ \times \mathbb{R}^d \to \mathbb{R}^{d\times \ell}$  satisfying  Assumption~\ref{assumption01}  and consider a  parameter function $\epsilon : \mathbb{R}_+ \to \R_+ $
   such that $\lim_{t\to \infty}\epsilon(t) = 0$ and $\int_0^\infty\epsilon(t)dt = \infty$.   Let $X_0$ be an $\mathcal{F}_0$-measurable, square-integrable random variable, and consider  the solution $x$ of \eqref{eqn_tikhonov-deterministic} starting from any $x_0\in \dom A$.   Then, the  Tikhonov regularized SDI \eqref{tikhonov-system} has an unique solution $X$ and it satisfies that   $$\displaystyle\lim_{t\to \infty}\mathbb{E}(\sup_{s\geq t}\|X_s-x(s)\|^2) = 0.$$
\end{theorem}

\begin{proof}
  First, the existence of the solution $(X,Y,M)$ is given by Theorem \ref{Theo:EU} by symply consider $F(t,x)=\epsilon(t) x$. Let us consider  $L:=\operatorname{span}(\dom A  - \dom A)$  by $P^\ast \colon \R^n \to L$ the linear proyection onto $L$. Second, by Itô's Lemma we have
    \begin{equation}\label{eqn-tikhonv2222-}
        \begin{aligned}
            \frac{1}{2}\|X_t-x(t)\|^2 = \ & \frac{1}{2}\|X_s-x(s)\|^2 + \int_s^t\langle X_\tau-x(\tau),dX_\tau-dx(\tau)\rangle\\
            &+\frac{1}{2}\int_s^t\|P^\ast\sigma(\tau,X_\tau)\|^2d\tau\\
            \leq \ & \frac{1}{2}\|X_s-x(s)\|^2 + \int_s^t\langle X_\tau-x(\tau), P^\ast\sigma(\tau,X_\tau)dB_\tau\rangle\\
            &-\int_s^t\epsilon(\tau)\|X_\tau-x(\tau)\|^2d\tau + \frac{1}{2}\int_s^t\sigma_\infty(\tau)^2d\tau\\
            &+\int_s^t\langle X_\tau-x(\tau),-dY_\tau+(-x'(\tau)-\epsilon(\tau)x(\tau))d\tau\rangle\\
            \leq \ & \frac{1}{2}\|X_s-x(s)\|^2 + \int_s^t\langle X_\tau-x(\tau), P^\ast\sigma(\tau,X_\tau)dB_\tau\rangle\\
            &-\int_s^t\epsilon(\tau)\|X_\tau-x(\tau)\|^2d\tau + \frac{1}{2}\int_s^t\sigma_\infty(\tau)^2d\tau.
        \end{aligned}
    \end{equation}
    where we have used the monotonicity by virtue of \cite[Proposition 6.17]{MR3308895}. Then, taking expectation we have 
    \begin{equation*}
        \begin{aligned}
            \mathbb{E}(\|X_t-x(t)\|^2)\leq \mathbb{E}(\|X_s-x(s)\|^2)-2\int_s^t\epsilon(\tau)\mathbb{E}(\|X_\tau-x(\tau)\|^2)d\tau+\int_s^t\mathbb{E}(\sigma_\infty(\tau)^2)d\tau.
        \end{aligned}
    \end{equation*}
    Then, by using Lemma \ref{lemma-comparison}, for all $t_0\in \R_+$ and $t\geq t_0$
    \begin{equation*}
        \begin{aligned}
        &\mathbb{E}(\|X_t-x(t)\|^2)\\
        \leq & \ \exp\left(-2\int_{t_0}^t\epsilon(\tau)d\tau\right)\left[\mathbb{E}(\|X_{t_0}-x(t_0)\|^2)+\int_{t_0}^t\exp\left(2\int_{t_0}^s\epsilon(\tau)d\tau\right)\mathbb{E}(\sigma_\infty(s)^2)ds\right]
        \end{aligned}
    \end{equation*}
    Replacing $t_0 = 0$ in the above inequality yields that $\mathbb{E}(\|X_t-x(t)\|^2)<\infty, \forall t\in \R_+$. Now, take $\eta>0$, then take $t_0$ big enough such that $\int_{t_0}^\infty\mathbb{E}(\sigma_\infty(\tau)^2)d\tau\leq \eta$, it follows that \begin{equation*}
        \mathbb{E}(\|X_t-x(t)\|^2)\leq \exp\left(-2\int_{t_0}^t\epsilon(\tau)d\tau\right)\mathbb{E}(\|X_{t_0}-x(t_0)\|^2) + \eta, \forall t\geq t_0.
    \end{equation*}
    It follows that $\displaystyle\limsup_{t\to \infty}\mathbb{E}(\|X_t-x(t)\|^2)\leq\eta$, and given that $\eta>0$ is arbitrary, we conclude that $$\displaystyle\lim_{t\to \infty}\mathbb{E}(\|X_t-x(t)\|^2)=0.$$
    Now, fix $t_0\in \R_+$, then define 
    \begin{equation*}
        \mathscr{M}_t:=\int_{t_0}^t\langle X_\tau-x(\tau),P^\ast \sigma(\tau,X_\tau)dB_\tau\rangle    
    \end{equation*} 
    Now, consider the stopping time $\tau_n := \inf\{\,t\geq r : \|X_t\|\geq n \,\}$.  It follows by \eqref{eqn-tikhonv2222-} that 
	\begin{equation}\label{Tyk01_stop2}
			\frac{1}{2}\|X_{t\wedge \tau_n}-x(t\wedge \tau_n)\|^2
			\leq \frac{1}{2}\|X_{t_0}-x(t_0)\|^2 + \sup_{t\geq t_0}|\mathscr{M}_{t\wedge\tau_n}|+ \frac{1}{2}\int_{t_0}^\infty\sigma_\infty(\tau)^2d\tau.
		\end{equation}
	By the same argument as in \eqref{eqBDG}, there exists a constant $\mathcal{K}>0$ such that, for all $n\in\N$,
	\begin{equation}\label{Tyk0223323}
		\mathbb{E}\Big(\sup_{t \geq t_0} |\mathscr{M}_{t\wedge\tau_n}|\Big) 
		\leq \frac{1}{4}\,\mathbb{E}\Big(\sup_{t\geq t_0}\|X_{t\wedge\tau_n}-x(t\wedge\tau_n)\|^2\Big) 
		+  \mathcal{K} \int_{t_0}^\infty \mathbb{E}(\sigma_\infty(\tau)^2) \, d\tau.
	\end{equation}
    Thus, taking expectation in \eqref{Tyk01_stop2} and using \eqref{Tyk0223323} yields
    \begin{equation*}           \mathbb{E}\Big(\sup_{t\geq t_0}\|X_{t\wedge\tau_n}-x(t\wedge\tau_n)\|^2\Big)\leq 2\mathbb{E}(\|X_{t_0}-x(t_0)\|^2) + (4\mathcal{K}+2)\int_{t_0}^\infty\mathbb{E}(\sigma_\infty(\tau)^2)d\tau.
    \end{equation*}
    Sending $n\to \infty$ in the above inequality, we can see that $\displaystyle \lim_{t\to \infty}\mathbb{E}(\sup_{s\geq t}\|X_s-x(s)\|^2) = 0$.
\end{proof}

Although the proof of Theorem~\ref{Main_resultTych} is based on estimating the discrepancy between $X_t$ and an auxiliary trajectory $x(t)$, the bounds obtained in the argument are not sharp enough to provide a quantitative rate for $\|X_t-x(t)\|$. As a consequence, even if one could derive a rate for $\|x(t)-x^\star\|$, this information cannot be directly converted into a rate of convergence for $\|X_t-x^\star\|$, nor for its expected value.

For this reason, by imposing a stronger assumption on the parameter function $\epsilon$, it becomes possible to compare the solution of the SDI~\eqref{tikhonov-system} with the curve $t\mapsto x_{\epsilon(t)}$, leading to a more explicit rate of convergence.

\begin{theorem}\label{thm-tikhonov}
 In the setting of Theorem~\ref{Main_resultTych}  let us suppose that the parameter function $\epsilon$ is $\mathcal{C}^1$ and satisfies $\displaystyle\lim_{t\to \infty}  \frac{\left|\epsilon'(t) \right|}{\epsilon(t)^2} = 0.  $   Then, the   solution $X$ of the  Tikhonov regularized SDI \eqref{tikhonov-system} satisfies that there exists a constant $C >0$ such that for all $r \geq 0$
 \begin{align}\label{T01}
 \mathbb{E}\Big(\sup_{s\geq r}\|X_{s }-x_{\epsilon(s )}\|^2\Big) 
		\leq C \left( \sup_{t\geq r}\frac{|\epsilon'(t)|}{\epsilon(t)^2} + \int_{r}^\infty \mathbb{E}(\sigma_\infty(\tau)^2) \, d\tau \right).
 \end{align}
\end{theorem}

\begin{proof} 
Let us consider  $L:=\operatorname{span}(\dom A  - \dom A)$  by $P^\ast \colon \R^n \to L$ the linear proyection onto $L$. By Itô's Lemma, we obtain
	\begin{equation*}
		\begin{aligned}
			\frac{1}{2}\|X_t-x_{\epsilon(t)}\|^2
			= \ &\frac{1}{2}\|X_s-x_{\epsilon(s)}\|^2 
			+ \int_s^t \langle X_\tau-x_{\epsilon(\tau)},-dY_\tau\rangle \\
			&- \int_s^t \epsilon(\tau)\langle X_\tau-x_{\epsilon(\tau)},X_\tau\rangle \,d\tau \\
			&+ \int_s^t \langle X_\tau-x_{\epsilon(\tau)},-\frac{d}{d\tau}x_{\epsilon(\tau)}\rangle\, d\tau 
			+ \frac{1}{2}\int_s^t \|P^\ast\sigma(\tau,X_\tau)\|^2 \, d\tau \\
			&+ \int_s^t \langle X_\tau-x_{\epsilon(\tau)},P^\ast\sigma(\tau,X_\tau)\, dB_\tau \rangle.
		\end{aligned}
	\end{equation*}
	Since $-\epsilon(\tau)x_{\epsilon(\tau)}\in A(x_{\epsilon(\tau)})$ and $\tau\mapsto x_{\epsilon(\tau)}$ is continuous, we have
	\begin{equation*}
		\int_s^t \langle X_\tau-x_{\epsilon(\tau)},-dY_\tau\rangle
		\leq \int_s^t \epsilon(\tau)\langle X_\tau-x_{\epsilon(\tau)},x_{\epsilon(\tau)}\rangle \, d\tau.
	\end{equation*}
	Now, fix $r >0$ and take $t,s >r$. Then,
	\begin{equation*}
		\begin{aligned}
			\frac{1}{2}\|X_t-x_{\epsilon(t)}\|^2
			\leq \ &\frac{1}{2}\|X_s-x_{\epsilon(s)}\|^2 
			-\int_s^t \epsilon(\tau)\|X_\tau-x_{\epsilon(\tau)}\|^2 \, d\tau \\
			&+ \int_s^t \|X_\tau-x_{\epsilon(\tau)}\|\,
			\Big\|\frac{d}{d\tau}x_{\epsilon(\tau)}\Big\| \, d\tau
			+ \frac{1}{2}\int_s^t \sigma_\infty(\tau)^2 \, d\tau \\
			&+ \int_s^t \langle X_\tau-x_{\epsilon(\tau)},P^\ast\sigma(\tau,X_\tau)\, dB_\tau \rangle \\
			\leq \ &\frac{1}{2}\|X_s-x_{\epsilon(s)}\|^2 \\
			&+ \int_s^t \epsilon(\tau)\|X_\tau-x_{\epsilon(\tau)}\|\,
			\big(\mathcal{L}_{r}-\|X_\tau-x_{\epsilon(\tau)}\|\big) \, d\tau \\
			&+ \frac{1}{2}\int_s^t \sigma_\infty(\tau)^2 \, d\tau
			+ \int_s^t \langle X_\tau-x_{\epsilon(\tau)},P^\ast\sigma(\tau,X_\tau)\, dB_\tau \rangle,
		\end{aligned}
	\end{equation*}
	where we have used \eqref{bound-derivative-curve1}, and
	\[
	\mathcal{L}_{r} := \|x^\star\| \sup_{t\geq r}\frac{|\epsilon'(t)|}{\epsilon(t)^2}.
	\]
	Define
	\[
	\mathcal{O}_r= \{\,t>r : \|X_t-x_{\epsilon(t)}\| > \mathcal{L}_{r} \,\}.
	\]
	Since $t \mapsto \|X_t-x_{\epsilon(t)}\|$ is continuous, $\mathcal{O}_r$ is open, then there is a countable family of pairwise disjoints open intervals $\{]a_i,b_i[\}_{i\in \N}$ such that $\mathcal{O}_r = \bigcup_{i\in \N}]a_i,b_i[$.  
	For $t\notin \mathcal{O}_r$ we have $\|X_t-x_{\epsilon(t)}\|\leq \mathcal{L}_{r}$, while for $t\in ]a_i,b_i[$ the condition $\|X_{a_i}-x_{\epsilon(a_i)}\|\leq \mathcal{L}_{r}$ implies
	\begin{equation*}
		\frac{1}{2}\|X_t-x_{\epsilon(t)}\|^2
		\leq \frac{1}{2}\mathcal{L}_{r}^2 
		+ \frac{1}{2}\int_{r}^t \sigma_\infty(\tau)^2 \, d\tau
		+ \int_{a_i}^t \langle X_\tau-x_{\epsilon(\tau)},P^\ast\sigma(\tau,X_\tau)\, dB_\tau \rangle .
	\end{equation*}
	Moreover, for all $t>a_i$,
	\[
	\int_{a_i}^t \langle X_\tau-x_{\epsilon(\tau)},P^\ast\sigma(\tau,X_\tau)\, dB_\tau \rangle 
	\leq 2 \sup_{s\in [r, t]} |\mathscr{M}_s|,
	\]
	where
	\[
	\mathscr{M}_s := \int_{r}^s \langle X_\tau-x_{\epsilon(\tau)},P^\ast\sigma(\tau,X_\tau)\, dB_\tau \rangle.
	\]
	Thus, for all $t>r$,
	\begin{equation}\label{Tyk01}
		\frac{1}{2}\|X_t-x_{\epsilon(t)}\|^2
		\leq \frac{1}{2}\mathcal{L}_{r}^2 
		+ \frac{1}{2}\int_{r}^t \sigma_\infty(\tau)^2 \, d\tau
		+ 2 \sup_{s\in [r,t]} |\mathscr{M}_s|.
	\end{equation}
	Now, consider the stopping time $\tau_n := \inf\{\,t\geq r : \|X_t\|\geq n \,\}$.  It follows by \eqref{Tyk01} that 
	\begin{equation}\label{Tyk01_stop}
			\frac{1}{2}\|X_{t\wedge \tau_n}-x_{\epsilon(t\wedge \tau_n)}\|^2
			\leq \frac{1}{2}\mathcal{L}_{r}^2 
			+ \frac{1}{2}\int_{r}^\infty \sigma_\infty(\tau)^2 \, d\tau
			+ 2 \sup_{s\geq r} |\mathscr{M}_{s\wedge \tau_n} |.
		\end{equation}
	By the same argument as in \eqref{eqBDG}, there exists a constant $\mathcal{K}>0$ such that, for all $n\in\N$,
	\begin{equation}\label{Tyk02}
		\mathbb{E}\Big(\sup_{s \geq r} |\mathscr{M}_{s\wedge\tau_n}|\Big) 
		\leq \frac{1}{8}\,\mathbb{E}\Big(\sup_{s\geq r}\|X_{s\wedge\tau_n}-x_{\epsilon(s\wedge\tau_n)}\|^2\Big) 
		+ 2 \mathcal{K} \int_{r}^\infty \mathbb{E}(\sigma_\infty(\tau)^2) \, d\tau.
	\end{equation}
	Taking expectation in  \eqref{Tyk01_stop},  and then using \eqref{Tyk02}, we obtain a constant $C>0$ such that, for all $n\in\N$,
	\begin{equation*}
		\frac{1}{4}\, \mathbb{E}\Big(\sup_{s\geq r}\|X_{s\wedge\tau_n}-x_{\epsilon(s\wedge\tau_n)}\|^2\Big) 
		\leq \frac{1}{2} \mathcal{L}_{r}^2 + C\int_{r}^\infty \mathbb{E}(\sigma_\infty(\tau)^2) \, d\tau.
	\end{equation*}
    Letting $n\to\infty$ we conclude that \eqref{T01}  holds,   which concludes the result.
\end{proof}
Let us emphasize that the estimate obtained in~\eqref{T01} shows that the convergence rate of the stochastic trajectory $(X_t)_{t\geq 0}$ is \emph{comparable} to that of the deterministic curve $x_{\epsilon(t)}$, up to constants that are intrinsic to the dynamics.
In particular, since $x_{\epsilon(t)}$ is deterministic, its rate of convergence toward $x^\star$ is independent of the random initial condition $X_0$.
This motivates the final corollary, where we derive explicit convergence rates as consequences of the above theorem in special cases of the operator $A$.

\begin{corollary}
Under the assumptions of  Theorem~\ref{thm-tikhonov},    define $z_t = \sup_{s\geq t}\frac{|\epsilon'(s)|}{\epsilon(s)^2}$.   Then, the following assertions hold
    \begin{enumerate}[label = (\alph*)]
        \item Suppose that $A\colon \R^d\tto \R^d$ is $L$-Lipschitz with $L>0$. Then, the solution $(X_t)$ of  satisfies the following
    \begin{equation}\label{rate-tikhonov1}
        \mathbb{E}\left(\|A(X_t)\|^2\right) = \mathcal{O}\left( \epsilon(t)^2  +  z_t^2  + \int_t^\infty\mathbb{E}(\sigma_\infty(s)^2)ds\right).
    \end{equation}
    \item Suppose that $A = \partial\varphi$, for a proper, convex and lower semicontinuous function  $\varphi\colon \R^d \to \R\cup\{ \infty\}$ and it satisfies the following error bound inequality on $\mathcal{S}$:
    \begin{equation*}
        \exists p\geq 1, \gamma>0, r>\min \varphi : \varphi(x)-\min \varphi\geq \gamma d(x;\mathcal{S})^p, \forall x\in [\varphi\leq r].
    \end{equation*}
    Thus,   we have
    \begin{equation}\label{rate-tikhonov2}
        \mathbb{E}\left(\|X_t-x^\star\|^2\right) = \mathcal{O}\left( \epsilon(t)^{1/p}  +  z_t^2  + \int_t^\infty\mathbb{E}(\sigma_\infty(s)^2)ds\right)
    \end{equation}
    \end{enumerate}
    
\end{corollary}
\begin{proof}
    Observe that, by \eqref{T01} we have  that 
    \begin{equation*}
        \begin{aligned}
        \mathbb{E}\Big(\sup_{s\geq t}\|X_{s}-x_{\epsilon(s)}\|^2\Big) 
		&\leq C \left( z_{t}^2  +  \int_{t}^\infty \mathbb{E}(\sigma_\infty(\tau)^2) \, d\tau \right)
        \end{aligned}
    \end{equation*}
    \emph{(a)}: Since $-\epsilon(s)x_{\epsilon(s)}\in A(x_{\epsilon(s)})$, we have that $$\|A(x_{\epsilon(s)})\| = \epsilon(s)\|x_{\epsilon(s)}\|\leq \epsilon(s)\|x^\star\|, \forall s\geq 0.$$
    Then,
    \begin{equation*}
        \begin{aligned}
            \mathbb{E}(\|A(X_t)\|^2) &\leq 2\mathbb{E}(\|A(X_t)-A(x_{\epsilon(t)})\|^2) + 2\mathbb{E}(\|A(x_{\epsilon(t)})\|^2)\\
            &\leq 2L^2\mathbb{E}(\|X_t-x_{\epsilon(t)}\|^2) + 2\|x^\star\|^2\epsilon(t)^2\\
            &\leq 2L^2C \left(z_t^2  +  \int_{t}^\infty \mathbb{E}(\sigma_\infty(\tau)^2) \, d\tau\right) + 2\|x^\star\|^2 \epsilon(t)^2,
        \end{aligned}
    \end{equation*}
    which allows to conclude \eqref{rate-tikhonov1}.\\
    \emph{(b)}: By using \cite[Proposition 23]{MR4991707} we have  there are constants $C^\ast, T^\ast>0$ such that $\mathbb{P}$-a.e.
    \begin{equation*}
        \|x_{\epsilon(t)}-x^\star\|^2\leq C^\ast \epsilon(t)^{1/p}, \forall t\geq T^\ast.
    \end{equation*}
    Thus, 
    \begin{equation*}
        \begin{aligned}
        \mathbb{E}\Big(\|X_{t}-x^\star\|^2\Big) &\leq 2\mathbb{E}\Big(\|X_{t}-x_{\epsilon(t)}\|^2\Big) + 2\mathbb{E}\Big(\| x_{\epsilon(t)} - x^\star\|^2\Big)\\
        &\leq 2 C  \left(  z_{t}^2  +  \int_{t}^\infty \mathbb{E}(\sigma_\infty(\tau)^2) \, d\tau \right) + 2C^\ast \epsilon(t)^{1/p}
        \end{aligned}
    \end{equation*}
    and then \eqref{rate-tikhonov2} follows.
\end{proof}

 \section{Conclusions}
 \label{sec:conclusions}
In this paper we were concerned with the asymptotic analysis of stochastic differential inclusions designed to solve monotone inclusion problems. Abandoning the classical non-empty interior assumption of the driving maximally monotone operator, we prove existence and uniqueness of solutions via a detailed constructed of orthogonal martingale processes. We then establish a finite time rate statement on the ergodic average of the process, in terms of gap functions associated with the monotone inclusion. Particularizing our model to the subgradient evolution case, asymptotic convergence to the minimum, as well as probabilistic concentration bounds are derived. Lastly, we study the strong convergence of the Tikhonov regularized subgradient flow. Our analysis reveals a slightly weakened Tikhonov regularization strategy is sufficient for ensuring strong convergence of the stochastic process to the least norm element. 

Many interesting directions for future research remain. First, it would be very interesting to study the long-time behavior of trajectories when the inclusion is not necessarily monotone, but rather hypomonotone \cite[Example 12.28]{MR1491362}. Additionally, understanding the multiscale aspects of the dynamical systems, and extension to the infinite-dimensional Hilbert space case are important next steps.

 \section*{Acknowledgments}
M. Staudigl acknowledges financial support from the Deutsche Forschungsgemeinschaft (DFG) -  Projektnummber 556222748 "non-stationary hierarchical minimization".

\bibliographystyle{plain}
\bibliography{references}

@incollection {MR957087,
    AUTHOR = {Moreau, J.-J.},
     TITLE = {Bounded variation in time},
 BOOKTITLE = {Topics in nonsmooth mechanics},
     PAGES = {1--74},
 PUBLISHER = {Birkh\"auser, Basel},
      YEAR = {1988},
      ISBN = {3-7643-1907-0},
   MRCLASS = {28B05 (26A45 46G99 46N05 70F99)},
  MRNUMBER = {957087},
MRREVIEWER = {Michel\ Valadier},
}

@article {MR4845874,
    AUTHOR = {Bo{\c t}, R. I. and Schindler, C.},
     TITLE = {On a stochastic differential equation with correction term
              governed by a monotone and Lipschitz continuous operator},
   JOURNAL = {Evol. Equ. Control Theory},
  FJOURNAL = {Evolution Equations and Control Theory},
    VOLUME = {14},
      YEAR = {2025},
    NUMBER = {3},
     PAGES = {463--493},
      ISSN = {2163-2472,2163-2480},
   MRCLASS = {60H10 (34D05 65K05 90C26 90C47)},
  MRNUMBER = {4845874},
MRREVIEWER = {Adrian\ Z\u alinescu},
       DOI = {10.3934/eect.2024064},
       URL = {https://doi.org/10.3934/eect.2024064},
}

@article {MR1442017,
    AUTHOR = {Bensoussan, A. and Rascanu, A.},
     TITLE = {Stochastic variational inequalities in infinite-dimensional
              spaces},
   JOURNAL = {Numer. Funct. Anal. Optim.},
  FJOURNAL = {Numerical Functional Analysis and Optimization. An
              International Journal},
    VOLUME = {18},
      YEAR = {1997},
    NUMBER = {1-2},
     PAGES = {19--54},
      ISSN = {0163-0563,1532-2467},
   MRCLASS = {34G20 (34A60 34F05 49J40)},
  MRNUMBER = {1442017},
MRREVIEWER = {John\ van der Hoek},
       DOI = {10.1080/01630569708816745},
       URL = {https://doi.org/10.1080/01630569708816745},
}

@article{luke2026asymptotic,
  title={Asymptotic behaviour of coupled random dynamical systems with multiscale aspects},
  author={Luke, D Russell and Schnebel, Johannes-Carl and Staudigl, Mathias and Peypouquet, Juan and Qu, Siqi},
  journal={arXiv preprint arXiv:2601.15411},
  year={2026}
}

@article {MR1626174,
    AUTHOR = {C\'epa, E.},
     TITLE = {Probl\`eme de {S}korohod multivoque},
   JOURNAL = {Ann. Probab.},
  FJOURNAL = {The Annals of Probability},
    VOLUME = {26},
      YEAR = {1998},
    NUMBER = {2},
     PAGES = {500--532},
      ISSN = {0091-1798,2168-894X},
   MRCLASS = {60H10 (60J60)},
  MRNUMBER = {1626174},
MRREVIEWER = {Constantin\ Tudor},
       DOI = {10.1214/aop/1022855642},
       URL = {https://doi.org/10.1214/aop/1022855642},
}

@article {MR3747470,
    AUTHOR = {Mertikopoulos, P. and Staudigl, M.},
     TITLE = {On the convergence of gradient-like flows with noisy gradient
              input},
   JOURNAL = {SIAM J. Optim.},
  FJOURNAL = {SIAM Journal on Optimization},
    VOLUME = {28},
      YEAR = {2018},
    NUMBER = {1},
     PAGES = {163--197},
      ISSN = {1052-6234,1095-7189},
   MRCLASS = {90C25 (60H10 90C15)},
  MRNUMBER = {3747470},
       DOI = {10.1137/16M1105682},
       URL = {https://doi.org/10.1137/16M1105682},
}

@article {MR4991707,
    AUTHOR = {Maulen-Soto, R. and Fadili, J. and Attouch, H.},
     TITLE = {Stochastic differential inclusions and {T}ikhonov
              regularization for stochastic non-smooth convex optimization
              in {H}ilbert spaces},
   JOURNAL = {Open J. Math. Optim.},
  FJOURNAL = {Open Journal of Mathematical Optimization (OJMO)},
    VOLUME = {6},
      YEAR = {2025},
     PAGES = {Art. No. 9, 30},
      ISSN = {2777-5860},
   MRCLASS = {90C25 (49J53 60H10 90C15 90C48)},
  MRNUMBER = {4991707},
}

@article {MR1398330,
    AUTHOR = {Attouch, H. and Cominetti, R.},
     TITLE = {A dynamical approach to convex minimization coupling
              approximation with the steepest descent method},
   JOURNAL = {J. Differential Equations},
  FJOURNAL = {Journal of Differential Equations},
    VOLUME = {128},
      YEAR = {1996},
    NUMBER = {2},
     PAGES = {519--540},
      ISSN = {0022-0396,1090-2732},
   MRCLASS = {90C48 (34E10 49L20)},
  MRNUMBER = {1398330},
MRREVIEWER = {J.\ Fr\'ed\'eric\ Bonnans},
       DOI = {10.1006/jdeq.1996.0104},
       URL = {https://doi.org/10.1006/jdeq.1996.0104},
}

@article {MR2462703,
    AUTHOR = {Cominetti, R. and Peypouquet, J. and Sorin, S.},
     TITLE = {Strong asymptotic convergence of evolution equations governed
              by maximal monotone operators with {T}ikhonov regularization},
   JOURNAL = {J. Differential Equations},
  FJOURNAL = {Journal of Differential Equations},
    VOLUME = {245},
      YEAR = {2008},
    NUMBER = {12},
     PAGES = {3753--3763},
      ISSN = {0022-0396,1090-2732},
   MRCLASS = {34G25 (34D10 47H05 47J35 47N20)},
  MRNUMBER = {2462703},
MRREVIEWER = {Nicolae\ H.\ Pavel},
       DOI = {10.1016/j.jde.2008.08.007},
       URL = {https://doi.org/10.1016/j.jde.2008.08.007},
}

@article {MR2731260,
    AUTHOR = {Peypouquet, J. and Sorin, S.},
     TITLE = {Evolution equations for maximal monotone operators: asymptotic
              analysis in continuous and discrete time},
   JOURNAL = {J. Convex Anal.},
  FJOURNAL = {Journal of Convex Analysis},
    VOLUME = {17},
      YEAR = {2010},
    NUMBER = {3-4},
     PAGES = {1113--1163},
      ISSN = {0944-6532,2363-6394},
   MRCLASS = {47J35 (34A60 34G10 47H05)},
  MRNUMBER = {2731260},
MRREVIEWER = {Teresa\ Winiarska},
}

@article {MR2593044,
    AUTHOR = {Attouch, H. and Czarnecki, M.-O.},
     TITLE = {Asymptotic behavior of coupled dynamical systems with
              multiscale aspects},
   JOURNAL = {J. Differential Equations},
  FJOURNAL = {Journal of Differential Equations},
    VOLUME = {248},
      YEAR = {2010},
    NUMBER = {6},
     PAGES = {1315--1344},
      ISSN = {0022-0396,1090-2732},
   MRCLASS = {34G25 (34D05 37B55 47N20 49J53 65J99)},
  MRNUMBER = {2593044},
MRREVIEWER = {Elena\ Resmerita},
       DOI = {10.1016/j.jde.2009.06.014},
       URL = {https://doi.org/10.1016/j.jde.2009.06.014},
}

@book {MR3308895,
	AUTHOR = {Pardoux, E. and R{\u a}{\c s}canu, A.},
	TITLE = {Stochastic differential equations, backward {SDE}s, partial
	differential equations},
	SERIES = {Stochastic Modelling and Applied Probability},
	VOLUME = {69},
	PUBLISHER = {Springer, Cham},
	YEAR = {2014},
	PAGES = {xviii+667},
	ISBN = {978-3-319-05713-2; 978-3-319-05714-9},
	MRCLASS = {60H05 (34F05 35D40 35R60 60H10 60J60)},
	MRNUMBER = {3308895},
	MRREVIEWER = {Mark\ A.\ McKibben},
	
	
}

@article {MR2577332,
	AUTHOR = {Voisei, M. D. and Z{\u a}linescu, C.},
	TITLE = {Maximal monotonicity criteria for the composition and the sum
	under weak interiority conditions},
	JOURNAL = {Math. Program.},
	FJOURNAL = {Mathematical Programming. A Publication of the Mathematical
	Programming Society},
	VOLUME = {123},
	YEAR = {2010},
	NUMBER = {1},
	PAGES = {265--283},
	ISSN = {0025-5610,1436-4646},
	MRCLASS = {47H05 (49J53 52A41)},
	MRNUMBER = {2577332},
	MRREVIEWER = {Radu\ Ioan\ Bo\c t},
	
	
}

@book {MR1121940,
	AUTHOR = {Karatzas, I. and Shreve, S. E.},
	TITLE = {Brownian motion and stochastic calculus},
	SERIES = {Graduate Texts in Mathematics},
	VOLUME = {113},
	EDITION = {Second},
	PUBLISHER = {Springer-Verlag, New York},
	YEAR = {1991},
	PAGES = {xxiv+470},
	ISBN = {0-387-97655-8},
	MRCLASS = {60J65 (35K99 35R60 60G44 60H10 60J60)},
	MRNUMBER = {1121940},
	
	
}

@book {MR2380366,
	AUTHOR = {Mao, X.},
	TITLE = {Stochastic differential equations and applications},
	EDITION = {Second},
	PUBLISHER = {Horwood Publishing Limited, Chichester},
	YEAR = {2008},
	PAGES = {xviii+422},
	ISBN = {978-1-904275-34-3},
	MRCLASS = {60-02 (34F05 34K50 60H10 60H30 91B28)},
	MRNUMBER = {2380366},
	
	
}

@book {MR348562,
	AUTHOR = {Br\'ezis, H.},
	TITLE = {Op\'erateurs maximaux monotones et semi-groupes de
	contractions dans les espaces de {H}ilbert},
	SERIES = {North-Holland Mathematics Studies},
	VOLUME = {No. 5},
	
	PUBLISHER = {North-Holland Publishing Co., Amsterdam-London; American
	Elsevier Publishing Co., Inc., New York},
	YEAR = {1973},
	PAGES = {vi+183},
	MRCLASS = {47H05},
	MRNUMBER = {348562},
	MRREVIEWER = {Bruce\ Calvert},
}

@article {MR2295146,
    AUTHOR = {Nesterov, Y.},
     TITLE = {Dual extrapolation and its applications to solving variational
              inequalities and related problems},
   JOURNAL = {Math. Program.},
  FJOURNAL = {Mathematical Programming. A Publication of the Mathematical
              Programming Society},
    VOLUME = {109},
      YEAR = {2007},
    NUMBER = {2-3},
     PAGES = {319--344},
      ISSN = {0025-5610,1436-4646},
   MRCLASS = {90C25 (49J40 90C47)},
  MRNUMBER = {2295146},
MRREVIEWER = {Vaithilingam\ Jeyakumar},
       DOI = {10.1007/s10107-006-0034-z},
       URL = {https://doi.org/10.1007/s10107-006-0034-z},
}

@article {MR3575644,
    AUTHOR = {Borwein, J. M. and Dutta, J.},
     TITLE = {Maximal monotone inclusions and {F}itzpatrick functions},
   JOURNAL = {J. Optim. Theory Appl.},
  FJOURNAL = {Journal of Optimization Theory and Applications},
    VOLUME = {171},
      YEAR = {2016},
    NUMBER = {3},
     PAGES = {757--784},
      ISSN = {0022-3239,1573-2878},
   MRCLASS = {49J52 (47H05 90C30)},
  MRNUMBER = {3575644},
MRREVIEWER = {Radu\ Ioan\ Bo\c t},
       DOI = {10.1007/s10957-015-0813-x},
       URL = {https://doi.org/10.1007/s10957-015-0813-x},
}

@book {MR3616647,
	AUTHOR = {Bauschke, H. H. and Combettes, P. L.},
	TITLE = {Convex analysis and monotone operator theory in {H}ilbert
	spaces},
	SERIES = {CMS Books in Mathematics/Ouvrages de Math\'ematiques de la
	SMC},
	EDITION = {Second},
	
	PUBLISHER = {Springer, Cham},
	YEAR = {2017},
	PAGES = {xix+619},
	ISBN = {978-3-319-48310-8; 978-3-319-48311-5},
	MRCLASS = {49-02 (41A65 46B20 46C05 47H05 90C25)},
	MRNUMBER = {3616647},
	
	
}

@article {MR3729487,
    AUTHOR = {Attouch, H. and Cabot, A. and Czarnecki,
              M.-O.},
     TITLE = {Asymptotic behavior of nonautonomous monotone and subgradient
              evolution equations},
   JOURNAL = {Trans. Amer. Math. Soc.},
  FJOURNAL = {Transactions of the American Mathematical Society},
    VOLUME = {370},
      YEAR = {2018},
    NUMBER = {2},
     PAGES = {755--790},
      ISSN = {0002-9947,1088-6850},
   MRCLASS = {34G25 (37N40 46N10 47H05)},
  MRNUMBER = {3729487},
MRREVIEWER = {Giovanni\ Colombo},
       DOI = {10.1090/tran/6965},
       URL = {https://doi.org/10.1090/tran/6965},
}

@article {MR4989859,
    AUTHOR = {Maul\'en S., R. and Fadili, J. and Attouch, H.},
     TITLE = {An {S}tochastic {D}ifferential {E}quation {P}erspective on
              {S}tochastic {C}onvex {O}ptimization},
   JOURNAL = {Math. Oper. Res.},
  FJOURNAL = {Mathematics of Operations Research},
    VOLUME = {50},
      YEAR = {2025},
    NUMBER = {4},
     PAGES = {3190--3221},
      ISSN = {0364-765X,1526-5471},
   MRCLASS = {90C25 (60H10 65C30)},
  MRNUMBER = {4989859},
       DOI = {10.1287/moor.2022.0162},
       URL = {https://doi.org/10.1287/moor.2022.0162},
}

@article {MR3361444,
    AUTHOR = {Combettes, P. L. and Pesquet, J.-C.},
     TITLE = {Stochastic quasi-{F}ej\'er block-coordinate fixed point
              iterations with random sweeping},
   JOURNAL = {SIAM J. Optim.},
  FJOURNAL = {SIAM Journal on Optimization},
    VOLUME = {25},
      YEAR = {2015},
    NUMBER = {2},
     PAGES = {1221--1248},
      ISSN = {1052-6234,1095-7189},
   MRCLASS = {47J25 (47H05 65K05 90C25 94A08)},
  MRNUMBER = {3361444},
MRREVIEWER = {Ramendra\ Krishna\ Bose},
       DOI = {10.1137/140971233},
       URL = {https://doi.org/10.1137/140971233},
}

@book {MR1491362,
    AUTHOR = {Rockafellar, R. T. and Wets, R. J.-B.},
     TITLE = {Variational analysis},
    SERIES = {Grundlehren der mathematischen Wissenschaften [Fundamental
              Principles of Mathematical Sciences]},
    VOLUME = {317},
 PUBLISHER = {Springer-Verlag, Berlin},
      YEAR = {1998},
     PAGES = {xiv+733},
      ISBN = {3-540-62772-3},
   MRCLASS = {49-02 (46N10 47N10 49J52 49K40 90C30)},
  MRNUMBER = {1491362},
MRREVIEWER = {Francis\ H.\ Clarke},
       DOI = {10.1007/978-3-642-02431-3},
       URL = {https://doi.org/10.1007/978-3-642-02431-3},
}

@book {MR3497465,
    AUTHOR = {Le Gall, J.-F.},
     TITLE = {Brownian motion, martingales, and stochastic calculus},
    SERIES = {Graduate Texts in Mathematics},
    VOLUME = {274},
   EDITION = {French},
 PUBLISHER = {Springer, [Cham]},
      YEAR = {2016},
     PAGES = {xiii+273},
      ISBN = {978-3-319-31088-6; 978-3-319-31089-3},
   MRCLASS = {60H05 (60G07 60G44 60H10 60J25 60J55 60J65)},
  MRNUMBER = {3497465},
       DOI = {10.1007/978-3-319-31089-3},
       URL = {https://doi.org/10.1007/978-3-319-31089-3},
}

@article {MR377609,
	AUTHOR = {Bruck, Jr., R. E.},
	TITLE = {Asymptotic convergence of nonlinear contraction semigroups in
	{H}ilbert space},
	JOURNAL = {J. Functional Analysis},
	FJOURNAL = {Journal of Functional Analysis},
	VOLUME = {18},
	YEAR = {1975},
	PAGES = {15--26},
	ISSN = {0022-1236},
	MRCLASS = {47H05 (46G05)},
	MRNUMBER = {377609},
	MRREVIEWER = {A.\ Pazy},
	
	
}

@book {MR755330,
	AUTHOR = {Aubin, J.-P. and Cellina, A.},
	TITLE = {Differential inclusions},
	SERIES = {Grundlehren der mathematischen Wissenschaften [Fundamental
	Principles of Mathematical Sciences]},
	VOLUME = {264},
	
	PUBLISHER = {Springer-Verlag, Berlin},
	YEAR = {1984},
	PAGES = {xiii+342},
	ISBN = {3-540-13105-1},
	MRCLASS = {49A50 (34A60 49E10)},
	MRNUMBER = {755330},
	MRREVIEWER = {S.\ Raczy\'nski},
	
	
}

@article {MR477932,
	AUTHOR = {Pazy, A.},
	TITLE = {On the asymptotic behavior of semigroups of nonlinear
	contractions in {H}ilbert space},
	JOURNAL = {J. Functional Analysis},
	FJOURNAL = {Journal of Functional Analysis},
	VOLUME = {27},
	YEAR = {1978},
	NUMBER = {3},
	PAGES = {292--307},
	ISSN = {0022-1236},
	MRCLASS = {47H99},
	MRNUMBER = {477932},
	MRREVIEWER = {Marshall\ Slemrod},
	
	
}

@incollection {MR2160665,
	AUTHOR = {Z{\u a}linescu, C.},
	TITLE = {A new proof of the maximal monotonicity of the sum using the
	{F}itzpatrick function},
	BOOKTITLE = {Variational analysis and applications},
	SERIES = {Nonconvex Optim. Appl.},
	VOLUME = {79},
	PAGES = {1159--1172},
	PUBLISHER = {Springer, New York},
	YEAR = {2005},
	ISBN = {978-0387-24209-5; 0-387-24209-0},
	MRCLASS = {49J53 (46N10 47H05)},
	MRNUMBER = {2160665},
	MRREVIEWER = {Giovanni\ Paolo\ Crespi},
	
	
}

@article {MR1774774,
	AUTHOR = {Pennanen, T.},
	TITLE = {Dualization of generalized equations of maximal monotone type},
	JOURNAL = {SIAM J. Optim.},
	FJOURNAL = {SIAM Journal on Optimization},
	VOLUME = {10},
	YEAR = {2000},
	NUMBER = {3},
	PAGES = {809--835},
	ISSN = {1052-6234,1095-7189},
	MRCLASS = {90C46 (47H04 47H05 47J05 49N15)},
	MRNUMBER = {1774774},
	
	
}
\end{document}